\documentclass[A4paper, 12pt]{article} 
\usepackage[T1]{fontenc}
\usepackage[utf8]{inputenc} 
\usepackage[all,cmtip]{xy}

\usepackage[margin=0.8in, top=1in, bottom=1in, footskip=0.5in]{geometry}
\usepackage{graphicx} 

\usepackage{booktabs} 
\usepackage{array} 
\usepackage{paralist} 
\usepackage{verbatim} 
\usepackage{subfig} 
\usepackage{mathtools}
\usepackage{amssymb}
\usepackage{amsmath}
\usepackage{amsthm}
\usepackage{mathrsfs}
\usepackage{enumitem}
\usepackage[colorlinks, allcolors=blue]{hyperref}
\usepackage{appendix}
\usepackage{stmaryrd}
\usepackage{hyperref}
\usepackage[bottom]{footmisc}
\hypersetup{colorlinks,linkcolor={blue},citecolor={blue},urlcolor={red}}

\usepackage{subfiles}

\usepackage{fancyhdr} 
\usepackage[tightpage,psfixbb]{preview}
\usepackage{sectsty}

\sectionfont{\normalsize}
\subsectionfont{\normalfont\itshape}
\subsubsectionfont{\normalfont\itshape}
\makeatletter

\usepackage[nottoc,notlof,notlot]{tocbibind} 
\usepackage[titles,subfigure]{tocloft} 
\theoremstyle{sltheoremstyle}
\newtheorem{theorem}{Theorem}[section]
\newtheorem{lemma}[theorem]{Lemma}
\newtheorem{corollary}[theorem]{Corollary}
\newtheorem{question}[theorem]{Question}
\newtheorem{proposition}[theorem]{Proposition}

\newtheorem*{corollary-non}{Corollary}
\newtheorem*{lemma-non}{Lemma}
\newtheorem*{theorem-non}{Theorem}
\newtheorem*{proposition-non}{Proposition}
\newtheorem*{condition-non}{Condition}
\newtheorem*{conditions-non}{Conditions}

\theoremstyle{definition}

\newtheorem{remark}[theorem]{Remark}
\newtheorem{remarks}[theorem]{Remarks}
\newtheorem{definition}[theorem]{Definition}

\usepackage[backend=bibtex,style=alphabetic]{biblatex}

\newcommand{\wh}[1]{{\widehat{#1}}}

\makeatletter
\newcommand{\extp}{\@ifnextchar^\@extp{\@extp^{\,}}}
\def\@extp^#1{\mathop{\bigwedge\nolimits^{\!#1}}}
\makeatother

\newcommand*{\img}[1]{%
    \raisebox{-.0\baselineskip}{%
        \includegraphics[
        height=\baselineskip,
        width=\baselineskip,
        keepaspectratio,
        ]{#1}%
    }%
}

\newcommand{\etale}{\textnormal{\'et}}

\newcommand{\CC}{\mathbb{C}}

\newcommand{\OO}{\mathcal{O}}

\newcommand{\QQ}{\mathbb{Q}}

\newcommand{\et}{\textnormal{\'et}}

\newcommand{\ch}{\textnormal{ch}}
\newcommand{\Hdg}{\textnormal{Hdg}}
\newcommand{\GL}{\textnormal{GL}}

\newcommand{\RR}{\mathbb{R}}

\newcommand{\ZZ}{\mathbb{Z}}

\newcommand{\CH}{\textnormal{CH}}

\newcommand{\Pic}{\textnormal{Pic}}

\newcommand{\Coker}{\textnormal{Coker}}
\newcommand{\Ker}{\textnormal{Ker}}
\newcommand{\NS}{\textnormal{NS}}

\newcommand{\id}{\textnormal{id}}

\newcommand{\Gal}{\textnormal{Gal}}

\newcommand{\Spec}{\textnormal{Spec }}

\newcommand{\Sp}{\textnormal{Sp}}

\newcommand{\ca}[1]{{\mathcal{#1}}}
\newcommand{\bb}[1]{{\mathbb{#1}}}
\newcommand{\msf}[1]{{\mathsf{#1}}}
\newcommand{\mr}[1]{{\mathscr{#1}}}
\newcommand{\mf}[1]{{\mathfrak{#1}}}
\newcommand{\tn}[1]{{\textnormal{#1}}}
\let\rm\relax 
\newcommand{\rm}[1]{{\mathrm{#1}}}

\DeclareSymbolFont{bbm}{U}{bbm}{m}{n}
\DeclareSymbolFontAlphabet{\mathbbm}{bbm}

\title{\textsc{\Large{Integral Fourier transforms and the integral Hodge conjecture for one-cycles on abelian varieties}}}
\author{Thorsten Beckmann $\quad \& \quad$ Olivier de Gaay Fortman}
\date{\small{\today}} 

\begin{document}

\newgeometry{left=0.6in, right=0.6in, top=0.8in, bottom=0.8in}

\pagestyle{plain}

\maketitle 

\abstract{\centering{\footnotesize{\noindent
We prove the integral Hodge conjecture for one-cycles on a principally polarized complex abelian variety whose minimal class is algebraic. In particular, the Jacobian of a smooth projective curve over the complex numbers satisfies the integral Hodge conjecture for one-cycles. The main ingredient is a lift of the Fourier transform to integral Chow groups. Similarly, we prove the integral Tate conjecture for one-cycles on the Jacobian of a smooth projective curve over the separable closure of a finitely generated field. Furthermore, abelian varieties satisfying such a conjecture are dense in their moduli space.
} 
}
}

\section{Introduction}
\label{introduction}


Let $g$ be a positive integer and let $A$ be an abelian variety of dimension $g$ over a field $k$ with dual abelian variety $\wh A$. The correspondence attached to the Poincar\'e bundle $\ca P_A$ on $A \times \wh A$ defines a powerful duality between the derived categories, rational Chow groups and cohomology of $A$ and $\wh A$ \cite{mukaiduality,beauvillefourier,huybrechtsfouriermukai}.
We shall refer to such morphisms as \textit{Fourier transforms}.

On the level of cohomology, the Fourier transform preserves integral $\ell$-adic \'etale cohomology when $k = k_s$ and integral Betti cohomology when $k = \CC$. It is thus natural to ask whether the Fourier transform on rational Chow groups preserves integral cycles modulo torsion or, more generally, lifts to a homomorphism between integral Chow groups. This question was raised by Moonen--Polishchuk \cite{moonenpolishchuk} and Totaro \cite{totaroIHCthreefolds}. More precisely, Moonen and Polishchuk gave a counterexample for abelian varieties over non-closed fields and asked about the case of algebraically closed fields.  

In this paper we further investigate this question with a view towards applications concerning the integral Hodge conjecture for one-cycles when $A$ is defined over $\CC$. To state our main result, we recall that whenever $\iota \colon C \hookrightarrow A$ is a smooth curve, the image of the fundamental class under the pushforward map $\iota_\ast \colon \rm H_{2}(C, \ZZ) \to \rm  H_{2}(A,\ZZ) \cong \rm  H^{2g-2}(A, \ZZ)$ defines a cohomology class $[C] \in \rm H^{2g-2}(A, \ZZ)$. This construction extends to one-cycles and factors modulo rational equivalence. As such, it induces a canonical homomorphism, called the \textit{cycle class map},
\[
cl\colon \CH_1(A) \to \textnormal{Hdg}^{2g-2}(A, \ZZ),
\]
which is a direct summand of a natural graded ring homomorphism $cl\colon \CH(A) \to \rm H^\bullet(A, \ZZ)$. 

The liftability of the Fourier transform turns out to have important consequences for the image of the cycle class map. Recall that an element $\alpha \in \rm H^{\bullet}(A, \ZZ)$ is called \textit{algebraic} if it is in the image of $cl$, and that $A$ satisfies the \textit{integral Hodge conjecture for $k$-cycles} if all Hodge classes in $\rm H^{2g-2k}(A,\ZZ)$ are algebraic. Although the integral Hodge conjecture fails in general \cite{atiyahintegralhodge, trento,totarocobordism}, it is an open question for abelian varieties. Our main result is as follows.

\newpage


\begin{theorem} \label{maintheorem}
Let $A$ be a complex abelian variety of dimension $g$ with Poincar\'e bundle $\ca P_A$. The following three statements are equivalent:
\begin{enumerate}[wide, labelwidth=!, labelindent=0pt]
    \item \label{introitem:minimalpoincare} The cohomology class $c_1(\ca P_A)^{2g-1}/(2g-1)! \in \rm H^{4g-2}(A \times \wh A, \ZZ)$ is algebraic. 
    \item \label{introitem:integralpoincare} The Chern character $\ch(\ca P_A) = \exp(c_1(\ca P_A)) \in \rm H^\bullet(A \times \wh A, \ZZ)$ is algebraic.
    \item \label{introitem:integralhodgeforproduct} The integral Hodge conjecture for one-cycles holds for $A \times \wh A$. 
\end{enumerate}
Any of these statements implies that
\begin{enumerate}[wide, labelwidth=!, labelindent=0pt]
\setcounter{enumi}{3}
    \item \label{introitem:IHC} The integral Hodge conjecture for one-cycles holds for $A$ and $\wh A$. 
\end{enumerate}
Suppose that $A$ is principally polarized by $\theta \in \Hdg^2(A,\ZZ)$ 
and consider the following statements:
\begin{enumerate}[wide, labelwidth=!, labelindent=0pt]
\setcounter{enumi}{4}
\item \label{introitem:minimalclass} The minimal cohomology class $\gamma_\theta \coloneqq \theta^{g-1}/(g-1)! \in \rm H^{2g-2}(A, \ZZ)$ is algebraic. 
\item \label{introitem:minimalpoincare2} The cohomology class $c_1(\ca P_A)^{2g-2}/(2g-2)! \in \rm H^{4g-4}(A \times \wh A, \ZZ)$ is algebraic. 
\item \label{introitem:last} For every algebraic cohomology class $\alpha \in \rm H^{> 0}(A, \ZZ)$ and every $i \in \ZZ_{\geq 1}$, the cohomology class $\alpha^i/i! \in \rm H^{\bullet}(A, \ZZ)$ is algebraic.  
\end{enumerate}
Then statements $\ref{introitem:minimalpoincare} - \ref{introitem:last}$ are equivalent. 
\end{theorem}
\noindent
Remark that Condition~\ref{introitem:minimalclass} is stable under products, so a product of principally polarized abelian varieties satisfies the integral Hodge conjecture for one-cycles if and only if each of the factors does. More importantly, if $J(C)$ is the Jacobian of a smooth projective curve $C$ of genus $g$, then 
every integral Hodge class of degree $2g-2$ on $J(C)$ is a $\ZZ$-linear combination of curves classes:

\begin{theorem} \label{introth:IHCforjacobians}
Let $C_1, \dotsc ,C_n$ be smooth projective curves over $\CC$. Then the integral Hodge conjecture for one-cycles holds for the product of Jacobians $J(C_1) \times \cdots \times J(C_n)$. 
\end{theorem}

\noindent
See Remark \ref{symmetricpowerremark}.\ref{symmetricremarkone} for another approach towards Theorem \ref{introth:IHCforjacobians} in the case $n=1$. A second consequence of Theorem \ref{maintheorem} is that the integral Hodge conjecture for one-cycles on principally polarized complex abelian varieties is stable under specialization, see Corollary \ref{complexspecialization}. An application of somewhat different nature is the following density result, proven in Section \ref{subsec:density}:

\begin{theorem} \label{introth:density}
Let $\delta = (\delta_1, \dotsc, \delta_g)$ be positive integers such that $\delta_i | \delta_{i+1}$ and let $\msf A_{g,\delta}(\CC)$ be the coarse moduli space of polarized abelian varieties over $\CC$ with polarization type $\delta$. There is a countable union $X\subset \msf A_{g,\delta}(\CC)$ of closed algebraic subvarieties of dimension at least $g$, that satisfies the following property: $X$ is dense in the analytic topology and the integral Hodge conjecture for one-cycles holds for those polarized abelian varieties whose isomorphism class lies in $X$. 
\end{theorem}
\begin{remark}
The lower bound that we obtain on the dimension of the components of $X$ actually depends on $\delta$ and is often greater than $g$. For instance, when $\delta = 1$ and $g\geq 2$, there is a set $X$ as in the theorem, whose elements are prime-power isogenous to products of Jacobians of curves. Therefore, the components of $X$ have dimension $3g-3$ in this case, c.f. Remark \ref{rem:dimensionimprovement}.
\end{remark}


One could compare Theorem \ref{maintheorem} with the following statement, proven by Grabowski \cite{grabowski}: if $g$ is a positive integer such that the minimal cohomology class $\gamma_\theta = \theta^{g-1}/(g-1)!$ of every principally polarized abelian variety of dimension $g$ is algebraic, then every abelian variety of dimension $g$ satisfies the integral Hodge conjecture for one-cycles. In this way, he proved the integral Hodge conjecture for abelian threefolds, a result which has been extended to smooth projective threefolds $X$ with $K_X = 0$ by Voisin and Totaro \cite{voisinIHCuniruled,totaroIHCthreefolds}. For abelian varieties of dimension greater than three, not many unconditional statements seem to have been known. 
\\
\\
The idea behind the proof of Theorem \ref{maintheorem} is the following. Let $A$ be a complex abelian variety of dimension $g$ and let $i \geq 0$ be an integer. Then Poincar\'e duality induces a canonical isomorphism $\varphi\colon \rm H^{2i}(A, \ZZ) \cong \rm H^{2g-2i}(A, \ZZ)^\vee \cong \rm H^{2g-2i}(\wh A, \ZZ)$. The map $\varphi$ respects the Hodge structures and thus induces an isomorphism $\Hdg^{2i}(A, \ZZ) \cong \Hdg^{2g-2i}(\wh A, \ZZ)$. However, it is unclear a priori whether $\varphi$ sends algebraic classes to algebraic classes. We prove that 
the algebraicity of $c_1(\ca P_A)^{2g-1}/(2g-1)! $ forces $\varphi$ to be algebraic, i.e.\ to be induced by a correspondence $\Gamma \in \CH(A \times \wh A)$. In particular, one then has $\rm Z^{2i}(A) \coloneqq \Hdg^{2i}(A, \ZZ) / \rm H^{2i}(A, \ZZ)_{\textnormal{alg}} \cong \rm Z^{2g-2i}(\wh A)$. To prove this, we lift the cohomological Fourier transform to a homomorphism between integral Chow groups whenever $c_1(\ca P_A)^{2g-1}/(2g-1)!$ is algebraic. For this we use a theorem of Moonen--Polishchuk saying that the ideal of positive dimensional cycles in the Chow ring with Pontryagin product of an abelian variety admits a divided power structure \cite[Theorem 1.6]{moonenpolishchuk}.
\\
\\
In Section \ref{sec:integralhodgeuptofactorn}, we consider an abelian variety $A_{/\CC}$ of dimension $g$ and ask: if $n \in \ZZ_{\geq 1}$ is such that $n \cdot c_1(\ca P_A)^{2g-1}/(2g-1)! \in \rm H^{4g-2}(A \times \wh A, \ZZ)_{\textnormal{alg}}$, can we bound the order of $\rm Z^{2g-2}(A)$ in terms of $g$ and $n$? For a smooth complex projective $d$-dimensional variety $X$, $\rm Z^{2d-2}(X)$ is called the degree $2d-2$ \textit{Voisin group} of $X$ \cite{perry2020integral}, is a stably birational invariant \cite[Lemma 2.20]{voisinstablyrational}, and related to the unramified cohomology groups by Colliot-Thélène--Voisin and Schreieder \cite{colliotthelenevoisin, schreieder2021refined}. 
We prove that if $n \cdot c_1(\ca P_A)^{2g-1}/(2g-1)!$ is algebraic, then $\gcd( n^2, (2g-2)!) \cdot \rm Z^{2g-2}(A) = (0)$. In particular, $(2g-2)! \cdot \rm Z^{2g-2}(A) = (0)$ for any $g$-dimensional complex abelian variety $A$.
Moreover, if $A$ is principally polarized by $\theta \in \textnormal{NS}(A)$ and if $n \cdot \gamma_\theta \in \rm H^{2g-2}(A,\ZZ)$ is algebraic, then $n \cdot c_1(\ca P_A)^{2g-1}/(2g-1)!$ is algebraic. Since it is well known that for Prym varieties, the Hodge class $2 \cdot \gamma_\theta$ is algebraic, these observations lead to the following result (see also Theorem~\ref{theorem:integralhodgeuptofactor}).
\begin{theorem} \label{introtheorem:prym}
Let $A$ be a $g$-dimensional Prym variety over $\CC$. Then $4 \cdot \rm Z^{2g-2}(A) = (0)$. 
\end{theorem}


For the study of the liftability of the Fourier transform, which was initiated by Moonen and Polishchuk in \cite{moonenpolishchuk}, it is more natural to consider abelian varieties defined over arbitrary fields. For this reason we define and study integral Fourier transforms in this generality, see Section \ref{sec:two}. We provide, for an abelian variety principally polarized by a symmetric ample line bundle, necessary and sufficient conditions for an integral Fourier transform to exist, see Theorem \ref{th:motivic}. 

This generality also allows to obtain the analogue of Theorem \ref{maintheorem} over the separable closure $k$ of a finitely generated field. Recall that a smooth projective variety $X$ of dimension $d$ over $k$ satisfies the \textit{integral Tate conjecture for one-cycles over $k$} if, for every prime number $\ell$ different from $\textnormal{char}(k)$ and for some finitely generated field of definition $k_0 \subset k$ of $X$, the cycle class map 
\begin{equation} \label{eq:integraltateconjecture}
    cl\colon \CH_1(X)_{\ZZ_\ell} = \CH_1(X)\otimes_\ZZ{\ZZ_\ell} \to \bigcup_U\rm H_{\textnormal{\'{e}t}}^{2d-2}(X, \ZZ_\ell(d-1))^U
\end{equation}
is surjective, where $U$ ranges over the open subgroups of $\Gal(k/k_0)$. 

\begin{theorem} \label{introth:integraltate} Let $A$ be an abelian variety of dimension $g$ over the separable closure $k$ of a finitely generated field. The following assertions are true:
\begin{enumerate}[wide, labelwidth=!, labelindent=0pt]
\item 
The abelian variety $A$ satisfies the integral Tate conjecture for one-cycles over $k$ if the cohomology class $$c_1(\ca P_A)^{2g-1}/(2g-1)! \in \rm H_{\textnormal{\'et}}^{4g-2}(A \times \wh A, \ZZ_\ell(2g-1))$$ is the class of a one-cycle with $\ZZ_\ell$-coefficients for every prime number $\ell < (2g-1)!$ unequal to $\textnormal{char}(k)$. 
    \item 
Suppose that $A$ is principally polarized and let $\theta_\ell \in \rm H^2_{\textnormal{\'et}}(A, \ZZ_\ell(1))$ be the class of the polarization. The map (\ref{eq:integraltateconjecture}) is surjective if $\gamma_{\theta_\ell} \coloneqq  \theta_\ell^{g-1}/(g-1)! \in \rm H_{\textnormal{\'{e}t}}^{2g-2}(A, \ZZ_\ell(g-1))$ is in its image. In particular, if $\ell > (g-1)!$ then this always holds. Thus $A$ satisfies the integral Tate conjecture for one-cycles if $\gamma_{\theta_\ell}$ is in the image of (\ref{eq:integraltateconjecture}) for every prime number $\ell < (g-1)!$ unequal to $\textnormal{char}(k)$. 
\end{enumerate}
\end{theorem}

Theorem \ref{introth:integraltate} implies that products of Jacobians of smooth projective curves over $k$ satisfy the integral Tate conjecture for one-cycles over $k$. Moreover, for a principally polarized abelian variety $A_K$ over a number field $K \subset \CC$, the integral Hodge conjecture for one-cycles on $A_\CC$ is equivalent to the integral Tate conjecture for one-cycles on $A_{\bar K}$ (Corollary \ref{hodgetatecomparison}), which in turn implies the integral Tate conjecture for one-cycles on the geometric special fiber $A_{\overline{k(\mf p})}$ of its N\'eron model of $A/\OO_K$ for any prime $\mf p \subset \OO_K$ at which $A_K$ has good reduction (Corollary \ref{cor:specialization}). 

Finally, Theorem \ref{introth:density} has an analogue in positive characteristic. The definition for a smooth projective variety over the algebraic closure $k$ of a finitely generated field to satisfy the \textit{integral Tate conjecture for one-cycles over $k$} is analogous to the definition above (see e.g.\ \cite{charlespirutka}). 
\begin{theorem} \label{th:chinglichai}
Let $k$ be the algebraic closure of a finitely generated field of characteristic $p>0$. 
Let $\msf A_{g}$ be the coarse moduli space over $k$ of principally polarized abelian varieties of dimension $g$ over $k$. The subset of $\msf A_{g}(k)$ of isomorphism classes of principally polarized abelian varieties over $k$ that satisfy the integral Tate conjecture for one-cycles over $ k$ is Zariski dense in $\msf A_{g}$. 
\end{theorem}


\textbf{Acknowledgements}. We are grateful for the encouragement and support of our respective PhD advisors Daniel Huybrechts and Olivier Benoist. We thank them, as well as Fabrizio Catanese and Frans Oort, for stimulating conversations. We thank Giuseppe Ancona, Olivier Benoist, Daniel Huybrechts and Burt Totaro for useful comments on an earlier draft of this paper. We thank the referee for his or her careful reading and for making valuable comments. 

The content of this article has been presented at the \textit{Research Seminar Algebraic Geometry} in Hanover, the \textit{Algebraic Geometry Seminar} in Berlin, and at the seminars \textit{Autour des cycles alg\'ebriques} and \textit{S\'eminaire de g\'eom\'etrie alg\'ebrique} in Paris during the months January and February 2022. We would like to thank the participants for ample feedback. 

We stress that work of Moonen and Polishchuk \cite{moonenpolishchuk} has been essential for our results. 

The first author is funded by the IMPRS program of the Max--Planck Society, the second by the European Union's Horizon 2020 research and innovation programme under the Marie Sk\l{}odowska-Curie grant agreement N\textsuperscript{\underline{o}} 754362 \img{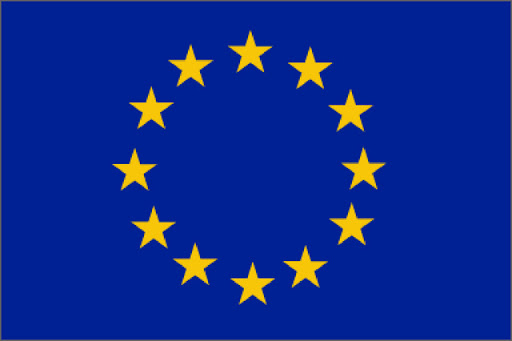}. The authors thank the Hausdorff Center for Mathematics as well as the \'Ecole normale sup\'erieure for their hospitality and pleasant working conditions during the respective stays of the authors which made this project possible. 


\section{Notation} \label{sec:notation}
\begin{itemize}[wide, labelwidth=!, labelindent=0pt]
\item 
We let $k$ be a field with separable closure $k_s$ and $\ell$ a prime number different from the characteristic of $k$. For a smooth projective variety $X$ over $k$, we let $\CH(X)$ be the Chow group of $X$ and define $\CH(X)_\QQ = \CH(X) \otimes \QQ$, $\CH(X)_{\QQ_\ell} = \CH(X) \otimes {\QQ_\ell}$ and $\CH(X)_{\ZZ_\ell} = \CH(X) \otimes {\ZZ_\ell}$. We let $\rm H_{\textnormal{\'et}}^i(X_{k_s}, \ZZ_\ell(a))$ be the $i$-th degree \'etale cohomology group with coeffients in $\ZZ_\ell(a)$, $a \in \ZZ$.  
\item 
Often, $A$ will denote an abelian variety of dimension $g$ over $k$, with dual abelian variety $\wh A$ and (normalized) Poincar\'e bundle $\ca P_A$ on $A \times \wh A$. The abelian group $\CH(A)$ will in that case be equipped with two ring structures: the usual intersection product $\cdotp$ as well as the Pontryagin product $\star$. Recall that the latter is defined as follows: 
\[\star \colon \CH(A) \times \CH(A) \to \CH(A), \quad x \star  y = m_{\ast} (\pi_1^{\ast}(x) \cdot \pi_2^{\ast}(y)). 
\]
Here, as well as in the rest of the paper, $\pi_i$ denotes the projection onto the $i$-th factor, $\Delta\colon A \to A \times A$ the diagonal morphism, and $m \colon A \times A \to A$ the group law morphism of $A$. There is a similar Pontryagin product $\star$ on \'etale cohomology, and on Betti cohomology if $k = \CC$. 
\item 
For any abelian group $M$ and any element $x \in M$, we will denote by $x_\QQ \in M \otimes_{\ZZ} \QQ$ the image of $x$ in $M \otimes_{\ZZ} \QQ$ under the canonical homomorphism $M \to M \otimes_\ZZ \QQ$. 
\end{itemize}

\section{Integral Fourier transforms and one-cycles on abelian varieties} \label{sec:two}


Our goal in this section is to provide necessary and sufficient conditions for the Fourier transform on rational Chow groups or cohomology to lift to a motivic homomorphism between integral Chow groups. We will relate such lifts to the integral Hodge conjecture when $k = \CC$. In subsequent Section \ref{sec:integralhodgeconjecture} we will use the theory developed in this section to prove Theorem \ref{maintheorem}. 

\subsection{Integral Fourier transforms and integral Hodge classes}
\label{sec:intfourierintHodge}

For abelian varieties $A$ over $k = k_s$, it is unknown whether the Fourier transform $\ca F_A \colon \CH(A)_{\QQ} \to \CH(\wh A)_{\QQ}$ preserves the subgroups of integral cycles modulo torsion. A sufficient condition for this to hold is that $\ca F_A$ lifts to a homomorphism $\CH(A) \to \CH(\wh A)$. In this section we outline a second consequence of such a lift $\CH(A) \to \CH(\wh A)$ when $A$ is defined over the complex numbers: the existence of an integral lift of $\ca F_A$ implies the integral Hodge conjecture for one-cycles on $\wh A$. 
\\
\\
Let $A$ be an abelian variety over $k$. The Fourier transform on the level of Chow groups is the group homomorphism
\[
\ca F_A \colon \CH(A)_\QQ \to \CH(\wh A)_\QQ
\]
induced by the correspondence $\ch(\ca P_A) \in \CH(A \times \wh A)_\QQ$, where $\ch(\ca P_A)$ is the Chern character of $\ca P_A$. Similarly, 
one defines the Fourier transform on the level of \'etale cohomology: 
\[
\mr F_{A}\colon \rm H_{\textnormal{\'{e}t}}^\bullet(A_{k_s}, \QQ_\ell(\bullet)) \to \rm H_{\textnormal{\'{e}t}}^\bullet(\wh A_{k_s}, \QQ_\ell(\bullet)). 
\]
In fact, $\mr F_A$ preserves the integral cohomology classes and induces, for each integer $j$ with $0 \leq j \leq 2g$, an isomorphism \cite[Proposition 1]{beauvillefourier}, \cite[page 18]{totaroIHCthreefolds}:
\[
    \mr F_{A}\colon \rm H_{\textnormal{\'{e}t}}^j(A_{k_s}, \ZZ_\ell(a)) \to \rm H_{\textnormal{\'{e}t}}^{2g-j}(\wh A_{k_s}, \ZZ_\ell(a+g-j)).
\]
Similarly, if $k = \CC$, then $\ch(\ca P_A)$ induces, for each integer $i$ with $0 \leq i \leq 2g$, an isomorphism of Hodge structures 
\begin{equation} \label{eq:integralfouriercohomology}
    \mr F_A\colon \rm H^i(A, \ZZ) \to \rm H^{2g-i}(\wh A, \ZZ)(g-i).
\end{equation}

In \cite{moonenpolishchuk}, Moonen and Polishchuk consider an isomorphism $\phi\colon A \xrightarrow{\sim} \wh A$, a positive integer $d$, and define the notion of motivic integral Fourier transform of $(A, \phi)$ up to factor $d$. The definition goes as follows. Let $\ca M(k)$ be the category of effective Chow motives over $k$ with respect to ungraded correspondences, and let $h(A)$ be the motive of $A$. Then a morphism $\ca F \colon h(A) \to h(A)$ in $\ca M(k)$ is a \textit{motivic integral Fourier transform of $(A, \phi)$ up to factor $d$} if the following three conditions are satisfied: (i) the induced morphism $h(A)_\QQ \to h(A)_\QQ$ is the composition of the usual Fourier transform with the isomorphism $\phi^\ast\colon h(\wh A)_\QQ \xrightarrow{\sim}h(A)_\QQ$, (ii) one has $d \cdot \ca F \circ \ca F = d \cdot (-1)^g \cdot [-1]_\ast$ as morphisms from $h(A)$ to $h(A)$, and (iii) $d\cdot \ca F \circ m_\ast = d \cdot \Delta^\ast \circ \ca F \otimes \ca F\colon h(A) \otimes h(A) \to h(A)$. 
\\
\\
For our purposes, we will consider similar homomorphisms $\CH(A) \to \CH(\wh A)$. However, to make their existence easier to verify (c.f.\ Theorem \ref{th:motivic}) we relax some of the above conditions:
 
\begin{definition} \label{def:weakintegralfourier}
Let $A_{/k}$ be an abelian variety and let $\ca F\colon \CH(A) \to \CH(\wh A)$ be a group homomorphism. We call $\ca F$ a \textit{weak integral Fourier transform} if the following diagram commutes:
\begin{equation}
\begin{split}
\xymatrix{
\CH(A) \ar[d]\ar[r]^{\ca F} & \CH(\wh A) \ar[d] \\
\CH(A)_\QQ \ar[r]^{\ca F_A} & \CH(\wh A)_\QQ.
}
\end{split}
\end{equation}
A group homomorphism $\ca F \colon \CH(A) \to \CH(\wh A)$ is an \textit{integral Fourier transform up to homology} if the following diagram commutes:
\begin{equation}\label{diagram:fouriercommutes}
\begin{split}
\xymatrixcolsep{5pc}
    \xymatrix{
\CH(A) \ar[d]\ar[r]^{\ca F} & \CH(\wh A) \ar[d] \\
\oplus_{r \geq 0} \rm H_{\textnormal{\'{e}t}}^{2r} (A_{k_s}, \ZZ_\ell(r)) \ar[r]^{\mr F_{A}} & \oplus_{r \geq 0} \rm H_{\textnormal{\'{e}t}}^{2r} (\wh A_{k_s}, \ZZ_\ell(r)).
}
\end{split}
\end{equation}
Similarly, a $\ZZ_\ell$-module homomorphism $\ca F_\ell \colon \CH(A)_{\ZZ_\ell} \to \CH(\wh A)_{\ZZ_\ell}$ is called an \textit{$\ell$-adic integral Fourier transform up to homology} if $\ca F_\ell$ is compatible with $\mr F_A$ and the $\ell$-adic cycle class maps. 
\end{definition}

\begin{remarks}
\begin{enumerate}[wide, labelwidth=!, labelindent=0pt]
    \item Let $\Gamma \in \CH(A \times \wh A)$ (resp.\ $\Gamma_\ell \in \CH(A \times \wh A)_{\ZZ_\ell})$ such that 
    \begin{align*}
cl(\Gamma) = \ch(\ca P_A) \quad \left(\tn{resp. } cl(\Gamma_\ell) = \ch(\ca P_A) \right)  \quad\tn{in} \quad \oplus_{r \geq 0}\rm H_{\textnormal{\'{e}t}}^{2r}((A \times \wh A)_{k_s}, \ZZ_\ell(r)). 
\end{align*}
Then
    $
    \ca F = \Gamma_\ast \colon \CH(A) \to \CH(\wh A)$ (resp. $\ca F_\ell = (\Gamma_\ell)_\ast \colon \CH(A)_{\ZZ_\ell} \to \CH(\wh A)_{\ZZ_\ell})$ is an integral Fourier transform up to homology (resp.\ an $\ell$-adic integral Fourier transform up to homology). Similarly, any cycle $\Gamma \in \CH(A \times \wh A)$ that satisfies $\Gamma_\QQ = \ch(\ca P_A) \in \CH(A \times \wh A)_{\QQ}$ induces a weak integral Fourier transform $\ca F = \Gamma_\ast \colon \CH(A) \to \CH(\wh A)$. 
    \item If $\ca F\colon \CH(A) \to \CH(\wh A)$ is a weak integral Fourier transform, then $\ca F$ is an integral Fourier transform up to homology, the $\ZZ_\ell$-module $\oplus_{r \geq 0}\rm H_{\textnormal{\'{e}t}}^{2r}(\wh A_{k_s}, \ZZ_\ell(r))$ being torsion-free. If $k = \CC$, then $\ca F\colon \CH(A) \to \CH(\wh A)$ is an integral Fourier transform up to homology if and only if $\ca F$ is compatible with the Fourier transform $\mr F_A\colon \rm H^\bullet(A, \ZZ) \to \rm H^\bullet(\wh A,\ZZ)$ on Betti cohomology. 
\end{enumerate}
\end{remarks} 

The relation between integral Fourier transforms and integral Hodge classes is as follows:

\begin{lemma}  \label{lemma:trivial}
Let $A$ be a complex abelian variety and $\ca F \colon \CH(A) \to \CH(\wh A)$ an integral Fourier transform up to homology. \begin{enumerate}[wide, labelwidth=!, labelindent=0pt]
\item For each $i \in \ZZ_{\geq 0}$, the integral Hodge conjecture for degree $2i$ classes on $A$ implies the integral Hodge conjecture for degree $2g-2i$ classes on $\wh A$.
\item If $\ca F = \Gamma_\ast$ for some $\Gamma \in \CH(A \times \wh A)$ with $cl(\Gamma) = \ch(\ca P_A)$, then $\mr F_A$ induces a group isomorphism $\rm Z^{2i}(A) \xrightarrow{\sim} \rm Z^{2g-2i}(\wh A)$ and, therefore, the integral Hodge conjectures for degree $2i$ classes on $A$ and degree $2g-2i$ classes on $\wh A$ are equivalent for all $i$. 
\end{enumerate}
\end{lemma}

\begin{proof}
We can extend Diagram (\ref{diagram:fouriercommutes}) to the following commutative diagram:
\begin{equation*} 
    \xymatrix{
\CH^{i}(A) \ar[d]^{cl^i} \ar[r] & \CH(A) \ar[d]\ar[r]^{\ca F} & \CH(\wh A) \ar[d] \ar[r] & \CH_i(\wh A) \ar[d]^{cl_i} \\
\rm H^{2i}(A, \ZZ) \ar[r] &\rm H^\bullet(A, \ZZ) \ar[r]^{\mr F_A} & \rm H^\bullet(\wh A, \ZZ) \ar[r] & \rm H^{2g-2i}(\wh A, \ZZ). 
}
\end{equation*}
The composition $\rm H^{2i}(A, \ZZ)  \to \rm H^{2g-2i}(\wh A, \ZZ)$ appearing on the bottom line agrees up to a suitable Tate twist with the map $\mr F_A$ of Equation (\ref{eq:integralfouriercohomology}). 
Therefore, we obtain a commutative diagram:
\begin{equation} \label{diagram:cycleclassmaps}
\begin{split}
\xymatrixcolsep{5pc}
\xymatrix{
\CH^i(A) \ar[d]^{cl^i}\ar[r] & \CH_i(\wh A) \ar[d]^{cl_i} \\
\Hdg^{2i}(A, \ZZ) \ar[r]^{\sim} & \Hdg^{2g-2i}(\wh A, \ZZ).
}
\end{split}
\end{equation}
Thus the surjectivity of $cl^i$ implies the surjectivity of $cl_i$. Moreover, if $\ca F$ is induced by some $\Gamma \in \CH(A \times \wh A)$, then replacing $A$ by $\wh A$ and $\wh A$ by $\skew{5.5}\widehat{\widehat{A}}$ in the argument above shows that the images of $cl^i$ and $cl_i$ are identified under the isomorphism $\mr F_A\colon \Hdg^{2i}(A, \ZZ) \xrightarrow{\sim} \Hdg^{2g-2i}(\wh A, \ZZ)$ in Diagram (\ref{diagram:cycleclassmaps}). 
\end{proof}

\subsection{Properties of the Fourier transform on rational Chow groups} \label{sec:propertiesfourier}

Let $A$ be a complex abelian variety. Observe that, for any $j \in \ZZ_{\geq 1}$ and $x \in \rm H^{2j}(A,\ZZ)$, one has
\[
\frac{x^i}{i!} \in \rm H^{2ij}(A,\ZZ) \subset \rm H^{2ij}(A, \QQ) \quad \tn{ for all } \quad i \in \ZZ_{\geq 1}. 
\]
In particular, the ideal 
$
\oplus_{j > 0}\rm H^{2j}(A, \ZZ) \subset \rm H^{2\bullet}(A, \ZZ)
$ 
admits a PD-structure \cite[\href{https://stacks.math.columbia.edu/tag/07GM}{Tag 07GM}]{stacks-project}. The analogue of this statement in $\ell$-adic étale cohomology holds when $A$ is an abelian variety over a separably closed field.

Lemma \ref{lemma:trivial} suggests that to prove Theorem \ref{maintheorem}, one would need to show that for a complex abelian variety of dimension $g$ whose minimal Poincar\'e class $c_1(\ca P_A)^{2g-1}/(2g-1)! \in \rm H^{4g-2}(A \times \wh A, \ZZ)$ is algebraic, all classes of the form $c_1(\ca P_A)^{i}/i! \in \rm H^{2i}(A \times \wh A, \ZZ)$ are algebraic. With this goal in mind we shall study Fourier transforms on rational Chow groups in Section \ref{sec:propertiesfourier}, and investigate how these relate to $\ch(\ca P_A) \in \CH(A \times \wh A)_{\QQ}$. In turns out that the cycles $c_1(\ca P_A)^{i}/i! \in \CH(A \times \wh A)_{\QQ}$ satisfy several relations that are very similar to those proved by Beauville for the cycles $\theta^{i}/i! \in \CH(A)_{\QQ}$ in case $A$ is principally polarized, see \cite{beauvillefourier}. Since we will need these results in any characteristic in order to prove Theorem \ref{introth:integraltate}, we will work over our general field $k$, see Section \ref{sec:notation}. 
\\
\\
Let $A$ be an abelian variety over $k$. Define 
\begin{align*}
\mr R_A  &= c_1(\ca P_A)^{2g-1}/(2g-1)! \in \CH_1(A\times \wh A)_{\QQ},\\
\mr R_{\wh A} &= c_1(\ca P_{\wh A})^{2g-1}/(2g-1)! \in \CH_1(\wh A\times A)_{\QQ}. 
\end{align*}
For $a \in \CH(A)_{\QQ}$, define $\mathrm{E}(a) \in \CH(A)_{\QQ}$ as the $\star$-exponential of $a$:
\[
\mathrm{E}(a)  \coloneqq \sum_{n \geq 0} \frac{a^{\star n}}{n!} \in \CH(A)_{\QQ},
\]
where $a^{\star n}$ denotes the $n$-fold Pontryagin product of $a$ (see Section \ref{sec:notation}). The key to Theorem \ref{maintheorem} will be the following:

\begin{lemma} \label{lemma:crucialprop} 
We have $\ch(\ca P_A) = \exp({c_1(\ca P_A)}) = (-1)^g\cdot \rm E((-1)^g\cdot \mr R_A) \in \CH( A \times \wh A)_{\QQ}$. 
\end{lemma}

\begin{proof} The most important ingredient in the proof is the following:
\\
\\
\hypertarget{claim1}{\textit{Claim \textcolor{blue}{$(\ast)$}}}: With respect to the Fourier transform $\ca F_{A \times \wh A}\colon \CH(A \times \wh A)_{\QQ} \to \CH(\wh A \times A)_{\QQ}$, one has $$\ca F_{A \times \wh A}\left(\exp(c_1(\ca P_A)) \right) = (-1)^g \cdot \exp( - c_1(\ca P_{\wh A}) ) \in \CH(\wh A \times A)_{\QQ}.$$
To prove Claim \textcolor{blue}{$(\ast)$}, we lift the desired equality in the rational Chow group of $\wh A \times A$ to an isomorphism in the derived category $\rm D^b(\wh A \times A)$ of $\wh A \times A$. For $X= A \times \wh A$ the Poincar\'e line bundle $\mathcal{P}_X$ on $X \times \wh X \cong A \times \wh A \times \wh A \times A$ is isomorphic to $\pi_{13}^\ast \mathcal{P}_A \otimes \pi_{24}^\ast \mathcal{P}_{\wh A}$. Let
\[
\Phi_{\ca P_X} \colon \rm D^b(A \times \wh A) \to \rm D^b(\wh A \times A)
\]
be the Fourier-Mukai transform attached to $\ca P_X \in \rm D^b(X \times \wh X)$ as in \cite[Definition 5.1]{huybrechtsfouriermukai}. Evaluating it at $\ca P_A$ gives the object
\begin{equation*}
\Phi_{{\ca P}_{X}}(\ca P_A) \cong \pi_{34,\ast} \left(\pi_{13}^\ast\ca P_A \otimes \pi_{24}^\ast \ca P_{\wh A} \otimes \pi_{12}^\ast\ca P_A  \right) \in \rm D^b(\wh A\times A),
\end{equation*}
whose Chern character is exactly $\mathcal{F}_X(\exp(c_1(\ca P_A)))$. Consider the permutation map
\[
(123) \colon A \times \wh A \times \wh A \times A \cong \wh A \times A \times \wh A \times A,
\]
with inverse $(321)$. We have
\begin{align*}
\pi_{34,\ast} \left(\pi_{13}^\ast\ca P_A \otimes \pi_{24}^\ast \ca P_{\wh A} \otimes \pi_{12}^\ast\ca P_A  \right)
&\cong 
\pi_{34,\ast} \left(
\pi_{31}^\ast\ca P_{\wh A} \otimes \pi_{12}^\ast\ca P_A
\otimes \pi_{24}^\ast \ca P_{\wh A}
\right) \\
& \cong 
\pi_{14,\ast} \left( (123)_\ast 
 \left(
\pi_{31}^\ast\ca P_{\wh A} \otimes \pi_{12}^\ast\ca P_A
\otimes \pi_{24}^\ast \ca P_{\wh A}
 \right)
 \right) \\
 & \cong 
 \pi_{14,\ast} \left( (321)^\ast 
 \left(
\pi_{31}^\ast\ca P_{\wh A} \otimes \pi_{12}^\ast\ca P_A
\otimes 
\pi_{24}^\ast \ca P_{\wh A}
 \right)
 \right) \\
 & \cong 
  \pi_{14,\ast} \left(
\pi_{12}^\ast\ca P_{\wh A}
\otimes 
\pi_{23}^\ast\ca P_A
\otimes 
\pi_{34}^\ast \ca P_{\wh A}
  \right).
\end{align*}
 We conclude that $\Phi_{{\ca P}_{X}}(\ca P_A)  \cong 
\pi_{14,\ast}\left( \pi_{12}^\ast\mathcal{P}_{\wh A} \otimes \pi_{23}^\ast \mathcal{P}_A \otimes \pi_{34}^\ast \mathcal{P}_{\wh A} \right)$. The latter is isomorphic to the Fourier--Mukai kernel of the composition
\[
\Phi_{\mathcal{P}_{\wh A}} \circ \Phi_{\mathcal{P}_A} \circ \Phi_{\mathcal{P}_{\wh A}}.
\]
Since $\Phi_{\mathcal{P}_{A}}\circ \Phi_{\mathcal{P}_{\wh A}}$ is isomorphic to $[-1_{\wh A}]^\ast \circ [-g]$ by \cite[Theorem 2.2]{mukaiduality}, we have
\[
\Phi_{\mathcal{P}_{\wh A}} \circ \Phi_{\mathcal{P}_A} \circ \Phi_{\mathcal{P}_{\wh A}}\cong  \Phi_{\mathcal{P}_{\wh A}} \circ [-1_{\wh A}]^\ast \circ [-g].
\]
This is the Fourier--Mukai transform with kernel $\mathcal{P}_{\wh A}^\vee[-g] \in \rm D^b(\wh A \times A)$. By uniqueness of the Fourier--Mukai kernel of an equivalence \cite[Theorem 2.2]{orlovfouriermukaikernel}, it follows that $\Phi_{\ca P_X}(\ca P_A) \cong \mathcal{P}_{\wh A}^\vee[-g]$. Since the Chern character of $\mathcal{P}_{\wh A}^\vee[-g]$ equals $(-1)^{g}\cdot \exp(-c_1(\ca P_{\wh A})) \in \CH(\wh A \times A)_{\QQ}$, this finishes the proof of Claim \hyperlink{claim1}{$(\ast)$}.
    \\
    \\
Next, we claim that $(-1)^g \cdot \ca F_{\wh A \times A}(-c_1(\ca P_{\wh A}))= \mr R_A$. 
To see this, recall that for each integer $i$ with $0 \leq i \leq g$, there is a canonical \textit{Beauville decomposition} $\CH^i(A)_\QQ = \oplus_{j = i-g}^i\CH^{i,j}(A)_{\QQ}$ \cite{beauvilledecomposition, deningermurre}. Since the Poincar\'e bundle $\ca P_A$ is symmetric, we have $c_1(\ca P_A) \in \CH^{1,0}(A \times \wh A)_{\QQ}$ and hence $c_1(\ca P_A)^{i} \in \CH^{i,0}(A \times \wh A)_{\QQ}$. In particular, we have $\mr R_A \in \CH^{2g-1,0}(A \times \wh A)_\QQ$. 
The fact that $\ca P_{A}$ is symmetric also implies - via Claim \hyperlink{claim1}{$(\ast)$} - that we have
$$
\ca F_{\wh A \times A}\left((-1)^g\cdot \exp(- c_1(\ca P_{\wh A})\right) = \exp(c_1(\ca P_A)).$$ Indeed, $\ca F_{\wh A \times A} \circ \ca F_{A \times \wh A} = [-1]^\ast \cdot (-1)^{2g} = [-1]^\ast$, see \cite[Corollary 2.22]{deningermurre}. Since $\ca F_{\wh A \times A}$ identifies $\CH^{i,0}(\wh A \times A)_{\QQ}$ with $\CH^{g-i,0}(A \times \wh A)$ (see \cite[Lemma 2.18]{deningermurre}), we must indeed have 
\begin{equation} \label{importantequation}
(-1)^g \cdot \ca F_{\wh A \times A}\left(-c_1(\ca P_{\wh A})\right)= \ca F_{\wh A \times A}\left((-1)^{g+1}\cdot c_1(\ca P_{\wh A})\right) =  \frac{c_1(\ca P_A)^{2g-1}}{(2g-1)!}  = \mr R_A. 
\end{equation}
For a $g$-dimensional abelian variety $X$ and any $x,y \in \CH(X)_{\QQ}$, one has $$\ca F_{X}(x \cdot y) = (-1)^g\cdot \ca F_{X}(x) \star \ca F_{X}(y)\in \CH(\wh X)_{\QQ}.$$ Indeed, in \cite[Theorem 4.5]{murre-cyclesonAVs} this is proved when $k$ is algebraically closed, but holds over general $k$ (and even for abelian schemes, see e.g. \cite{geeredixmoon}). 

This implies (see also \cite[\S 3.7]{moonenpolishchuk}) that if $a$ is a cycle on $X$ such that $\ca F_{X}(a) \in \CH_{>0}(\wh X)_{\QQ}$, then $\ca F_{X}(\exp(a)) =
(-1)^g\cdot \rm E((-1)^g\cdot\ca F_{X}(a))$. This allows us to conclude that
\begin{align*}
    \exp(c_1(\ca P_A)) &= 
    \ca F_{\wh A \times A}
    \left( (-1)^g \cdot \exp(-c_1(\ca P_{\wh A}))\right)  \\
    &= (-1)^{g} \cdot\rm E\left(\ca F_{\wh A \times A}(-c_1(\ca P_{\wh A}))\right) \\
    &= (-1)^g\cdot \rm E((-1)^g\cdot \mr R_A),
\end{align*}
which finishes the proof. 
\end{proof}

Next, assume that $A$ is equipped with a \textit{principal} polarization $\lambda \colon A \xrightarrow{\sim} \wh A$, 
and define
\begin{equation} \label{generaltheta}
    \Theta = \frac{1}{2}\cdot (\id, \lambda)^\ast c_1(\ca P_A) \in \CH^1(A)_\QQ \quad  \tn{ and } \quad    \widehat\Theta = \frac{1}{2}\cdot (\lambda^{-1}, \id)^\ast c_1(\ca P_{A}) \in \CH^1(\wh A)_\QQ. 
\end{equation}
Here $(\id,\lambda)$ (resp. $(\lambda^{-1}, \id)$) is the morphism 
$(\id,\lambda) \colon A \to A\times \wh A$ (resp. $(\lambda^{-1}, \id) \colon \wh A \to A \times \wh A$). One can understand the relation between 
\[
\Gamma_\Theta \coloneqq \Theta^{g-1}/(g-1)! \in \CH_1(A)_{\QQ}
\]
and $\mr R_A = {c_1(\ca P_A)}^{2g-1}/(2g-1)! \in \CH_1(A \times \wh A)_{\QQ}$ in the following way. 
Define $j_1\colon A \to A \times \wh A$ and $j_{2} \colon \wh A \to A \times \wh A$ by $x \mapsto (x,0)$ and $y \mapsto (0,y)$ respectively. 
Define a one-cycle $\tau$ on $A \times \wh A$ as follows:
\begin{align*} 
\tau \coloneqq 
j_{1, \ast} (\Gamma_\Theta)  
+  j_{2, \ast}(\Gamma_{\wh \Theta}) - (\id, \lambda)_\ast (\Gamma_\Theta) 
\in \CH_1( A \times \wh A )_\QQ.
\end{align*}

\begin{lemma} \label{lemma:minimalclasspoincarecomparison}
One has 
$
 \tau = (-1)^{g+1}\cdot \mr R_A \in \CH_1(A \times \wh A)_{\QQ}$. 
\end{lemma}

\begin{proof}


Identify $A$ and $\wh A$ via $\lambda$. This gives $c_1(\ca P_A) = m^\ast(\Theta) - \pi_1^\ast(\Theta) - \pi_2^\ast(\Theta)$, and the Fourier transform becomes an endomorphism $\ca F_A \colon \CH(A)_\QQ \to \CH(A)_\QQ$. We claim that $\tau = (-1)^g \cdot \left(  \Delta_\ast \ca F_A(\Theta) - j_{1, \ast} \ca F_A(\Theta) -  j_{2, \ast} \ca F_A(\Theta) \right)$. For this, it suffices to show that $\ca F_A(\Theta) = (-1)^{g-1}\cdot \Theta^{g-1}/(g-1)! \in \CH_1(A)_\QQ$. Now $\ca F_A(\exp(\Theta)) = \exp({-\Theta})$ by Lemma \ref{beauvillelemma} below. Moreover, since $\Theta$ is symmetric, we have $\Theta \in \CH^{1,0}(A)_{\QQ}$, hence $\Theta^{i}/i! \in \CH^{i,0}(A)_\QQ$ for each $i \geq 0$. Therefore, $\ca F_A\left(\Theta^{i}/i!\right) \in \CH^{g-i,0}(A)_\QQ$ by \cite[Lemma 2.18]{deningermurre}. This implies that in fact, $\ca F_A\left(\Theta^i/i!\right) =(-1)^{g-i}\cdot \Theta^{g-i}/(g-i)! \in \CH^{g-i,0}(A)_\QQ$ for every $i$. In particular, the claim follows. 
\\
\\
Next, recall that $\ca F_{A\times A}(c_1(\ca P_A)) = (-1)^{g+1}\cdot \mr R_A$, see Claim \hyperlink{claim1}{$(\ast)$}. So at this point, it suffices to prove the identity $\ca F_{A\times A}(c_1(\ca P_A)) = (-1)^g \cdot \left(  \Delta_\ast \ca F_A(\Theta) - j_{1, \ast} \ca F_A(\Theta) -  j_{2, \ast} \ca F_A(\Theta) \right)$. To prove this, we use the following functoriality properties of the Fourier transform on the level of rational Chow groups. Let $X$ and $Y$ be abelian varieties and let $f\colon X \to Y$ be a homomorphism with dual homomorphism $\wh f\colon \wh Y \to \wh X$. We then have the following equalities \cite[(3.7.1)]{moonenpolishchuk}: 
\begin{equation}\label{eq:functoriality}
    (\wh f)^\ast \circ \ca F_X = \ca F_Y \circ f_\ast, \hspace{3mm}  \ca F_X \circ f^\ast = (-1)^{\dim X - \dim Y}\cdot (\wh f)_\ast \circ \ca F_Y. 
\end{equation}
Since $c_1(\ca P_A) = m^\ast \Theta - \pi_1^\ast \Theta - \pi_2^\ast \Theta$, it follows from Equation (\ref{eq:functoriality}) that
\begin{align*}
    \ca F_{A \times A}(c_1(\ca P_A)) &=  \ca F_{A \times A}\left( m^\ast \Theta\right) - \ca F_{A \times A}\left(\pi_1^\ast \Theta\right) - \ca F_{A \times A}\left(\pi_2^\ast \Theta\right) \\
 &= (-1)^g\cdot \left(\Delta_\ast \ca F_A(\Theta) - j_{1, \ast} \ca F_A(\Theta) -  j_{2, \ast} \ca F_A(\Theta) \right).
\end{align*}
\end{proof}

\begin{lemma}[Beauville] \label{beauvillelemma}
Let $A$ be an abelian variety over $k$, principally polarized by $\lambda \colon A \xrightarrow{\sim} \wh A$, and define $\Theta = \frac{1}{2}\cdot (\id, \lambda)^\ast c_1(\ca P_A) \in \CH^1(A)_{\QQ}$. Identify $A$ and $\wh A$ via $\lambda$. With respect to the Fourier transform $\ca F_A\colon \CH(A)_{\QQ} \xrightarrow{\sim} \CH(A)_{\QQ}$, one has $\ca F_A(\exp(\Theta)) = \exp({-\Theta})$. 
\end{lemma}

\begin{proof}
Our proof follows the proof of \cite[Lemme 1]{beauvillefourier}, but has to be adapted, since $\Theta$ does not necessarily come from a symmetric ample line bundle on $A$. Since one still has $c_1(\ca P_A) = m^\ast \Theta - \pi_1^\ast \Theta - \pi_2^\ast \Theta$, the argument can be made to work: one has  
\begin{align*}
\ca F_A(\exp(\Theta)) &= 
\pi_{2,\ast}\left(
\exp(c_1(\ca P_A)) \cdot \pi_1^\ast \exp({\Theta}) 
\right) \\
&=
\pi_{2,\ast}\left(\exp({m^\ast \Theta - \pi_2^\ast\Theta}) \right)\\ &= \exp({-\Theta})  \cdot \pi_{2,\ast}(m^\ast \exp({\Theta})) \quad \quad \in \;  \CH(A)_{\QQ}.
\end{align*}
For codimension reasons, one has $\pi_{2,\ast}(m^\ast \exp(\Theta)) = \pi_{2,\ast}m^\ast(\Theta^g/g!)= \deg(\Theta^g/g!) \in \CH^0(A)_{\QQ} \cong \QQ$. Pull back $\Theta^g/g!$ along $A_{k_s} \to A$ to see that $\deg(\Theta^g/g!) = 1 \in \CH^0(A)_{\QQ} \cong \CH^0(A_{k_s})_{\QQ}$, since over $k_s$ the cycle $\Theta$ becomes the cycle class attached to a symmetric ample line bundle.
\end{proof}

\subsection{Divided powers and integral Fourier transforms}




It was asked by Bruno Kahn whether there exists a PD-structure on the Chow ring of an abelian variety over any field with respect to its usual (intersection) product. There are counterexamples over non-closed fields: see \cite{esnaultelementarytheorems}, where Esnault constructs an abelian surface $X$ and a line bundle $\ca L$ on $X$ such that $c_1(\ca L) \cdot c_1(\ca L)$ is not divisible by $2$ in $\CH_0(X)$. However, the case of algebraically closed fields remains open \cite[Section 3.2]{moonenpolishchuk}. What we do know, is the following:

\begin{theorem}[Moonen--Polishchuk] \label{th:PDstructure}
Let $A$ be an abelian variety over $k$. The ring $\left(\CH(A), \star \right)$ admits a canonical PD-structure $\gamma$ on the ideal $\CH_{>0}(A) \subset \CH(A)$. If $k = \bar k$, then $\gamma$ extends to a PD-structure on the ideal generated by $\CH_{>0}(A)$ and the zero cycles of degree zero. 
\end{theorem}

In particular, for each element $x \in \CH_{>0}(A)$ and each $n \in \ZZ_{\geq 1}$, there is a canonical element $x^{[n]} \in \CH_{>0}(A)$ such that $n!x^{[n]} = x^{\star n}$, see \cite[\href{https://stacks.math.columbia.edu/tag/07GM}{Tag 07GM}]{stacks-project}. For $x \in \CH_{>0}(A)$, we may then define $\mathrm{E}(x) = \sum_{n \geq 0} x^{[n]} \in \CH(A)$ as the $\star$-exponential of $x$ in terms of its divided powers. 
\\
\\
Together with the results of Section \ref{sec:propertiesfourier}, Theorem \ref{th:PDstructure} enables us to provide several criteria for the existence of a weak integral Fourier transform. We recall that for an abelian variety $A$ over $k$, principally polarized by $\lambda \colon A \xrightarrow{\sim} \wh A$, we defined $\Theta \in \CH^1(A)_\QQ$ to be the symmetric ample class attached to the polarization $\lambda$, see Equation (\ref{generaltheta}). 

\newpage

\begin{theorem} \label{th:motivic}
Let $A_{/k}$ be an abelian variety of dimension $g$. The following are equivalent:
\begin{enumerate}[wide, labelwidth=!, labelindent=0pt]
    \item \label{motivicone} The one-cycle $\mr R_A = c_1(\ca P_A)^{2g-1}/(2g-1)! \in \CH(A \times \wh A)_{\QQ}$ lifts to a one-cycle in $\CH(A \times \wh A)$. 
    \item \label{motivictwo} The cycle $\ch(\ca P_A) \in \CH(A \times \wh A)_\QQ$ lifts to a cycle in $\CH(A \times \wh A)$. 
    \item \label{motivicthree} 
    The cycle $\ch(\ca P_{A \times \wh A}) \in \CH(A \times \wh A \times \wh A \times A)_\QQ$ lifts to a cycle in $\CH(A \times \wh A \times \wh A \times A)$. 
\end{enumerate}
Moreover, if $A$ carries a symmetric ample line bundle that induces a principal polarization $\lambda \colon A \xrightarrow{\sim} \wh A$, then 
the above statements are equivalent to the following equivalent statements:
\begin{enumerate}[wide, labelwidth=!, labelindent=0pt]
\setcounter{enumi}{3}
    \item \label{motiviczero} The two-cycle $c_1(\ca P_A)^{2g-2}/(2g-2)! \in \CH(A \times \wh A)_{\QQ}$ lifts to a two-cycle in $\CH(A \times \wh A)$.
    \item \label{motivicfour} The one-cycle $\Gamma_\Theta  = \Theta^{g-1}/(g-1)! \in \CH(A)_\QQ$ lifts to a one-cycle in $\CH(A)$. 
    \item \label{motivicfive} The abelian variety $A$ admits a weak integral Fourier transform. 
    \item \label{motivicsix} The Fourier transform $\ca F_A$ 
    satisfies $\ca F_A\left( \CH(A)/\textnormal{torsion}\right) \subset \CH(\wh A)/\textnormal{torsion}$.  
    \item \label{motiviceight} There exists a PD-structure on the ideal $\CH^{>0}(A)/\textnormal{torsion} \subset \CH(A)/\textnormal{torsion}$. 
\end{enumerate}
\end{theorem}

\begin{proof}
Suppose that \ref{motivicone} holds, and let $\Gamma \in \CH_1(A \times \wh A)$ be a cycle such that $\Gamma_\QQ = \mr R_A$. Then consider the cycle $(-1)^g\cdot \rm E((-1)^g\cdot\Gamma) \in \CH(A \times \wh A)$. By Lemma \ref{lemma:crucialprop}, we have
\[
(-1)^g\cdot \rm E((-1)^g\cdot \Gamma)_{\QQ} = (-1)^g\cdot \rm E((-1)^g\cdot \Gamma_{\QQ}) = (-1)^g\cdot \rm E((-1)^g\cdot \mr R_A) = \ch(\ca P_A) \in \CH(A \times \wh A)_{\QQ}.
\]
Thus \ref{motivictwo} holds. We claim that \ref{motivicthree} holds as well. Indeed, consider the line bundle $\ca P_{A \times \wh A}$ on the abelian variety $X= A \times \wh A \times \wh A \times A$; one has that $\ca P_{A \times \wh A} \cong \pi_{13}^\ast\ca P_A \otimes \pi_{24}^\ast \ca P_{\wh A}$, which implies that 
{\footnotesize \begin{equation}\label{eq1}
\begin{split}
\mr R_{A \times \wh A} &= \frac{1}{(4g-1)!}\cdot \left( 
\pi_{13}^\ast c_1(\ca P_A) + \pi_{24}^\ast c_1(\ca P_{\wh A})
\right)^{4g-1} \\
&= 
\frac{1}{(2g)!(2g-1)!} \cdot 
\left(
\pi_{13}^\ast c_1(\ca P_A)^{2g-1} \cdot \pi_{24}^\ast c_1(\ca P_{\wh A})^{2g}
+
\pi_{13}^\ast c_1(\ca P_A)^{2g} \cdot \pi_{24}^\ast c_1(\ca P_{\wh A})^{2g-1}
\right) \\
&= 
\frac{1}{(2g)!(2g-1)!} \cdot 
\left(
\pi_{13}^\ast c_1(\ca P_A)^{2g-1} \cdot \pi_{24}^\ast \left((2g)! \cdot [0]_{A \times \wh A}\right)
+
\pi_{13}^\ast \left(
(2g)! \cdot [0]_{\wh A \times A}
\right)
\cdot \pi_{24}^\ast c_1(\ca P_{\wh A})^{2g-1}
\right) \\
&= 
\pi_{13}^\ast \left( \frac{c_1(\ca P_A)^{2g-1}}{(2g-1)!} \right)\cdot \pi_{24}^\ast ([0]_{A \times \wh A})
+
\pi_{13}^\ast (
[0]_{\wh A \times A}
)
\cdot \pi_{24}^\ast \left( \frac{c_1(\ca P_{\wh A})^{2g-1}}{(2g-1)!} \right) \quad 
 \in \; \CH_1(X)_\QQ.
 \end{split}
\end{equation}}
We conclude that $\mr R_{A \times \wh A}$ lifts to $\CH_1(X)$ which, by the implication $[\ref{motivicone} \implies \ref{motivictwo}]$ (that has already been proved), implies that \ref{motivicthree} holds. 
The implication $[\ref{motivicthree}\implies\ref{motivicone}]$ follows from the fact that $(-1)^g \cdot \ca F_{\wh A \times A}(- c_1(\ca P_{\wh A})) = \mr R_A$ (see Equation (\ref{importantequation})). 
Therefore, we have $[\ref{motivicone} \iff \ref{motivictwo} \iff \ref{motivicthree}]$. 

Let us from now on assume that $A$ is principally polarized by $\lambda \colon A \xrightarrow{\sim} A$, where $\lambda$ is the polarization attached to a symmetric ample line bundle $\ca L$ on $A$. Moreover, in what follows we shall identify $\wh A$ and $A$ via $\lambda$. 

Suppose that \ref{motiviczero} holds and let $S_{ A} \in \CH_2(A\times  A) = \CH^{2g-2}(A \times A)$ be such that 
$$(S_{A})_\QQ = c_1(\ca P_A)^{2g-2}/(2g-2)! \in \CH_2(A \times A)_\QQ.
$$
Define $\CH^{1,0}(A) \coloneqq \Pic^{\textnormal{sym}}(A)$ to be the group of isomorphism classes of symmetric line bundes on $A$. Then $S_{A}$ induces a homomorphism $\ca F \colon \CH^{1,0}(A) \to \CH_1( A)$ defined as the composition 
\[
\ca F \colon \CH^{1,0}(A) \xrightarrow{\pi_1^\ast} \CH^1(A \times  A) \xrightarrow{\cdot S_{A}} \CH^{2g-1}(A \times A) = \CH_1(A \times A) \xrightarrow{\pi_{2,\ast}} \CH_1(A).
\]
Since $\ca F_A \left(\CH^{1,0}(A)_\QQ\right) \subset \CH_1( A)_\QQ$ (see \cite[Lemma 2.18]{deningermurre}) we see that the diagram 
\begin{equation}
\begin{split}
\label{commmuttttivitiy}
\xymatrix{
\CH^{1,0}(A)\ar[d] \ar[r]^{\ca F} & \CH_1( A)\ar[d] \\
\CH^{1,0}(A)_\QQ \ar[r]^{\ca F_A} & \CH_1( A)_{\QQ}
}
\end{split}
\end{equation}
commutes. On the other hand, since the line bundle $\ca L$ is symmetric, we have 
\begin{equation}\label{liftoftheta}
\Theta = \frac{1}{2}\cdot \Delta^\ast c_1(\ca P_A) = \frac{1}{2}\cdot  c_1\left( \Delta^\ast\ca P_A\right)  = \frac{1}{2} \cdot c_1(\ca L \otimes \ca L) = c_1(\ca L) \in \CH^1(A)_{\QQ}. 
\end{equation}
The class $c_1(\ca L) \in \CH^{1,0}(A)$ of the line bundle $\ca L$ thus lies above $\Theta \in \CH^1(A)_\QQ$. Therefore, $\ca F(c_1(\ca L)) \in \CH_1(A)$ lies above $\Gamma_\Theta = (-1)^{g-1}\ca F_{A}(\Theta)$ by the commutativity of (\ref{commmuttttivitiy}), and \ref{motivicfour} holds.

Suppose that \ref{motivicfour} holds. Then \ref{motivicone} follows readily from Lemma \ref{lemma:minimalclasspoincarecomparison}. Moreover, if \ref{motivictwo} holds, then $\ch(\ca P_A) \in \CH(A \times A)_\QQ$ lifts to $\CH(A \times A)$, hence in particular \ref{motiviczero} holds. Since we have already proved that \ref{motivicone} implies \ref{motivictwo}, we conclude that $[\ref{motiviczero} \implies  \ref{motivicfour} \implies \ref{motivicone} \implies \ref{motivictwo} \implies \ref{motiviczero}]$. 

The implications $[\ref{motivictwo}\implies\ref{motivicfive} \implies \ref{motivicsix}]$ are trivial. Assume that \ref{motivicsix} holds. By Equation (\ref{liftoftheta}), $\Theta \in \CH^1(A)_{\QQ}$ lifts to $\CH^1(A)$, hence $\ca F_{A}(\Theta) = (-1)^{g-1}\cdot \Gamma_\Theta$ lifts to $\CH_1(A)$, i.e.\ \ref{motivicfour} holds. 

Assume that \ref{motivicsix} holds. The fact that $    \ca F_{A}\left( 
    \CH(A)/\textnormal{torsion} 
    \right) \subset \CH(A)/\textnormal{torsion}$ implies that 
\begin{equation*}
    \CH(A)/\textnormal{torsion} = \ca F_{A}\left( \ca F_{ A} \left( 
    \CH( A)/\textnormal{torsion} 
    \right) \right) 
    \subset 
    \ca F_{A}\left( 
    \CH(A)/\textnormal{torsion} 
    \right) \subset \CH(A)/\textnormal{torsion}.
\end{equation*}
Thus, the restriction of the Fourier transform $\ca F_{A}$ to $\CH(A)/\textnormal{torsion}$ defines an isomorphism $$\ca F_{A}\colon \CH(A)/\textnormal{torsion} \xrightarrow{\sim} \CH( A)/\textnormal{torsion}.$$ Now if $R$ is a ring and $\gamma$ is a PD-structure on an ideal $I \subset R$, then $\gamma$ extends to a PD-structure on $I/\textnormal{torsion} \subset R/\textnormal{torsion}$. Consequently, the ideal $\CH_{>0}(A)/\textnormal{torsion} \subset \CH(A)/\textnormal{torsion}$ admits a PD-structure for the Pontryagin product $\star$ by Theorem \ref{th:PDstructure}. Since $\ca F_A$ exchanges the Pontryagin and intersection product (up to a sign, see \cite[Proposition~3(ii)]{beauvillefourier}), it follows that \ref{motiviceight} holds. Since \ref{motiviceight} trivially implies \ref{motivicfour}, we are done. 
\end{proof}

\begin{question}[Moonen--Polishchuk \cite{moonenpolishchuk}, Totaro \cite{totaroIHCthreefolds}] \label{questionmoonenpolishchuk}
Let $A$ be any principally polarized abelian variety over $k = \bar k$. Are the equivalent conditions in Theorem \ref{th:motivic} satisfied for $A$? 
\end{question}

\begin{remark}
For Jacobians of hyperelliptic curves the answer to Question \ref{questionmoonenpolishchuk} is "yes" \cite{moonenpolishchuk}. 
\end{remark}

Similarly, there is a relation between integral Fourier transforms up to homology and the algebraicity of minimal cohomology classes induced by Poincar\'e line bundles and theta divisors. 

\begin{proposition} \label{prop:motivicetale-new}
Let $A_{/k}$ be an abelian variety of dimension $g$. The following are equivalent:
\begin{enumerate}[wide, labelwidth=!, labelindent=0pt]
    \item \label{motivicone-new} The class $c_1(\ca P_A)^{2g-1}/(2g-1)! \in \rm H^{4g-2}_{\etale}((A \times \wh A)_{k_s}, \ZZ_\ell(2g-1))$ lifts to $\CH_1(A \times \wh A)$. 
    \item \label{motivictwo-new} The class $\ch(\ca P_A) \in \oplus_{r \geq 0}\rm H_{\textnormal{\'{e}t}}^{2r}((A \times \wh A)_{k_s}, \ZZ_\ell(r))$ lifts to a cycle in $\CH(A \times \wh A)$.  
    \item \label{motivicthree-new} 
   The class $$\ch(\ca P_{A \times \wh A}) \; \in \; \oplus_{r \geq 0}\rm H_{\textnormal{\'{e}t}}^{2r}((A \times \wh A \times \wh A \times A)_{k_s}, \ZZ_\ell(r))$$ lifts to a cycle in $\CH(A \times \wh A \times \wh A \times A)$. 
\end{enumerate}
Moreover, if $A$ carries an ample line bundle that induces a principal polarization $\lambda \colon A \xrightarrow{\sim} \wh A$, then 
the above statements are equivalent to the following equivalent statements:
\begin{enumerate}[wide, labelwidth=!, labelindent=0pt]
\setcounter{enumi}{3}
    \item \label{motiviczero-new} The class $ c_1(\ca P_A)^{2g-2}/(2g-2)! \in \rm H_{\etale}^{4g-4}((A \times \wh A)_{k_s}, \ZZ_\ell(2g-2))$ lifts to $\CH_2(A \times \wh A)$.
    \item \label{motivicfour-new} The class $\gamma_\theta  = \theta^{g-1}/(g-1)! \in \rm H^{2g-2}_{\etale}(A_{k_s}, \ZZ_\ell(g-1))$ lifts to a cycle in $\CH_1(A)$. 
    \item \label{motivicfive-new} The abelian variety $A$ admits an integral Fourier transform up to homology.  
    \item \label{motivicsix-new} The Fourier transform $\mr F_A$ 
    satisfies $\mr F_A\left( \rm H^{2\bullet}_{\et}(A_{k_s}, \ZZ_\ell(\bullet))_{\tn{alg}}\right) \subset \rm H^{2\bullet}_{\et}(\wh A_{k_s}, \ZZ_\ell(\bullet))_{\tn{alg}}$.  
    \item \label{motiviceight-new} There exists a PD-structure on the ideal 
    $
\oplus_{j > 0} \rm H^{2j}_{\et}(A_{k_s}, \ZZ_\ell(j))_{\tn{alg}} \subset 
\rm H^{2\bullet}_{\et}(A_{k_s}, \ZZ_\ell(\bullet))_{\tn{alg}}$. 
\end{enumerate}
Here, $\rm H_{\textnormal{\'{e}t}}^{2\bullet}(A_{k_s}, \ZZ_\ell(\bullet))_{\tn{alg}}$ denotes the 
image of the cycle class map 
    $
\CH^\bullet(A) \to \rm H_{\textnormal{\'{e}t}}^{2\bullet}(A_{k_s}, \ZZ_\ell(\bullet)). 
    $
\end{proposition}

\begin{proof}[Proof of Proposition \ref{prop:motivicetale-new}]
The proof of Theorem \ref{th:motivic} can easily be adapted to this situation.  
\end{proof}

\newpage

\begin{proposition} \label{remarketalebetti}
\begin{enumerate}[wide, labelwidth=!, labelindent=0pt]
\item \label{remark3.12.1}
If $k = \CC$, then each of the statements $\ref{motivicone-new} - \ref{motiviceight-new}$ in Proposition \ref{prop:motivicetale-new} is equivalent to the same statement with \'etale cohomology replaced by Betti cohomology. 
\item \label{remark3.12.2}
Proposition \ref{prop:motivicetale-new} remains valid if one replaces integral Chow groups by their tensor product with $\ZZ_\ell$, 
`integral Fourier transform up to homology' by `$\ell$-adic integral Fourier transform up to homology', and $\rm H_{\textnormal{\'{e}t}}^{2\bullet}(A_{k_s}, \ZZ_\ell(\bullet))_{\tn{alg}}$ 
by the image of the map $
\CH^\bullet(A) \otimes \ZZ_\ell \to \rm H_{\textnormal{\'{e}t}}^{2\bullet}(A_{k_s}, \ZZ_\ell(\bullet)). 
    $
\end{enumerate}
\end{proposition}
\begin{proof}
\begin{enumerate}[wide, labelwidth=!, labelindent=0pt]
\item 
In this case $\ZZ_\ell(i) = \ZZ_\ell$ and the Artin comparison isomorphism \[\rm H^{2i}_{\textnormal{\'et}}(A, \ZZ_\ell) \xrightarrow{\sim} \rm H^{2i}(A(\CC), \ZZ_\ell)\] \cite[III, Expos\'e XI]{SGA4} is compatible with the cycle class map. Since the map $\rm H^{2i}(A(\CC), \ZZ) \to \rm H_{\textnormal{\'et}}^{2i}(A, \ZZ_\ell)$ is injective, a class $\beta \in \rm H^{2i}(A(\CC), \ZZ)$ is in the image of $cl\colon \CH^i(A) \to \rm H^{2i}(A(\CC),\ZZ)$ if and only if its image $\beta_\ell \in \rm H^{2i}_{\textnormal{\'et}}(A, \ZZ_\ell)$ is in the image of $cl\colon \CH^i(A) \to \rm H^{2i}_{\textnormal{\'et}}(A, \ZZ_\ell)$. 
\item 
Indeed, for an abelian variety $A$ over $k$, the PD-structure on $\CH_{>0}(A) \subset (\CH(A), \star)$ induces a PD-structure on $\CH_{>0}(A) \otimes \ZZ_\ell \subset (\CH(A)_{\ZZ_\ell}, \star)$ by \cite[\href{https://stacks.math.columbia.edu/tag/07H1}{Tag 07H1}]{stacks-project}, because the ring map $(\CH(A),\star) \to (\CH(A)_{\ZZ_\ell},\star)$ is flat. The latter follows from the flatness of $\ZZ \to \ZZ_\ell$.
\end{enumerate}
\end{proof}


\section{The integral Hodge conjecture for one-cycles on complex abelian varieties} \label{sec:integralhodgeconjecture}

In this section we use the theory developed in Section \ref{sec:two} to prove Theorem \ref{maintheorem}. We also prove some applications of Theorem \ref{maintheorem}: the integral Hodge conjecture for one-cycles on products of Jacobians (Theorem \ref{introth:IHCforjacobians}), the fact that the integral Hodge conjecture for one-cycles on principally polarized complex abelian varieties is stable under specialization (Corollary \ref{complexspecialization}) and density of polarized abelian varieties satisfying the integral Hodge conjecture for one-cycles (Theorem \ref{introth:density}). 

\subsection{Proof of the main theorem} \label{sec:proofmaintheorem}

Let us prove Theorem \ref{maintheorem}. 

\begin{proof}[Proof of Theorem \ref{maintheorem}]
Suppose that \ref{introitem:minimalpoincare} holds. Then \ref{introitem:integralpoincare} holds by Propositions \ref{prop:motivicetale-new} and \ref{remarketalebetti}.\ref{remark3.12.1}. Suppose that \ref{introitem:integralpoincare} holds. Then 
\ref{introitem:IHC} follows from Lemma \ref{lemma:trivial}. So we have $[\ref{introitem:minimalpoincare} \iff \ref{introitem:integralpoincare}  \implies \ref{introitem:IHC}]$. 

For a complex abelian variety $X$ of dimension $g$, define
\[
\rho_X = c_1(\ca P_X)^{2g-1}/(2g-1)! \in \rm H^{4g-2}(X \times \wh X, \ZZ). 
\]
If $\ref{introitem:minimalpoincare}$ holds, then $\rho_A = c_1(\ca P_A)^{2g-1}/(2g-1)! \in \rm H^{4g-2}(A \times \wh A, \ZZ)$ is algebraic, which implies that $\rho_{\wh A} \in \rm H^{4g-2}(\wh A \times A, \ZZ)$ is algebraic. Therefore, $$\rho_{A \times \wh A} 
\in \rm H^{8g-2}(A \times \wh A \times \wh A \times A, \ZZ)$$ is algebraic by Equation (\ref{eq1}). We then apply the implication $[\ref{introitem:minimalpoincare}  \implies \ref{introitem:IHC}]$ to the abelian variety $A \times \wh A$, which shows that \ref{introitem:integralhodgeforproduct} holds. Since $[\ref{introitem:integralhodgeforproduct}  \implies \ref{introitem:minimalpoincare}]$ is trivial, we have proven $[\ref{introitem:minimalpoincare} \iff \ref{introitem:integralpoincare}  \iff \ref{introitem:integralhodgeforproduct} \implies  \ref{introitem:IHC}]$. 

Next, assume that $A$ is principally polarized by $\theta \in \textnormal{NS}(A) \subset \rm H^2(A, \ZZ)$. The directions $[\ref{introitem:IHC}\implies \ref{introitem:minimalclass}]$ and $[\ref{introitem:integralpoincare} \implies \ref{introitem:minimalpoincare2}]$ are trivial and $[\ref{introitem:minimalclass} \implies \ref{introitem:minimalpoincare}]$ follows from Propositions \ref{prop:motivicetale-new} and \ref{remarketalebetti}.\ref{remark3.12.1}. We claim that \ref{introitem:minimalpoincare2} implies \ref{introitem:IHC}. Define $\sigma_A = c_1(\ca P_A)^{2g-2}/(2g-2)! \in \rm H^{4g-4}(A \times \wh A,\ZZ)$ and let $S \in \CH_2(A \times \wh A)$ be such that $cl(S) =  \sigma_A$. The squares in the following diagram commute:
\begin{equation} \label{diagramwithonlysurfaceclasses}
\begin{split}
\xymatrixcolsep{4pc}
\xymatrix{
\CH^1(A) \ar[d]^{cl}\ar[r]^{\pi_1^\ast}& \CH^1(A \times \wh A) \ar[d]^{cl}\ar[r]^{\cdot S}& \CH^{2g-1}(A \times \wh A) \ar[d]^{cl} \ar[r]^{\pi_{2,\ast}} & \CH_1(\wh A) \ar[d]^{cl} \\
\rm H^2(A,\ZZ) \ar[r]^{\pi_1^\ast} & \rm H^{2}(A \times \wh A,\ZZ) \ar[r]^{\cdot \sigma_A} & \rm H^{4g-2}(A \times \wh A,\ZZ) \ar[r]^{\pi_{2,\ast}} & \rm H^{2g-2}(\wh A,\ZZ). 
}
\end{split}
\end{equation}
Since $\mr F_A = \pi_{2,\ast}\left( \ch(\ca P_A) \cdot \pi_1^\ast(-)\right)$ 
restricts to an isomorphism $\mr F_A\colon \rm H^2(A,\ZZ) \xrightarrow{\sim} \rm H^{2g-2}(\wh A,\ZZ)$ by \cite[Proposition 1]{beauvillefourier}, the composition $\pi_{2,\ast} \circ (- \cdot \sigma_A) \circ \pi_1^\ast$ on the bottom row of (\ref{diagramwithonlysurfaceclasses}) is an isomorphism. Thus, by Lefschetz $(1,1)$, $cl\colon \CH_1(\wh A) \to \Hdg^{2g-2}(\wh A,\ZZ)$ is surjective. 

Finally, the equivalence $[\ref{introitem:minimalclass} \iff \ref{introitem:last}]$ follows directly from Propositions \ref{prop:motivicetale-new} and \ref{remarketalebetti}.\ref{remark3.12.1}. 
\end{proof}

\begin{corollary} \label{cor:IHCforproducts}
Let $A$ and $B$ be complex abelian varieties of respective dimensions $g_A$ and $g_B$. 
\begin{enumerate}[wide, labelwidth=!, labelindent=0pt]
    \item The Hodge classes $$c_1(\ca P_A)^{2g_A-1}/(2g_A-1)! \in \rm H^{4g_A-2}(A \times \wh A, \ZZ) \quad \tn{ \emph{and} } \quad c_1(\ca P_B)^{2g_B-1}/(2g_B-1)! \in \rm H^{4g_B-2}(B \times \wh B, \ZZ)$$ are algebraic if and only if $A \times \wh A$, $B \times \wh B$, $A\times B$ and $\wh A \times \wh B$ satisfy the integral Hodge conjecture for one-cycles. 
    \item If $A$ and $B$ are principally polarized, then the integral Hodge conjecture for one-cycles holds for $A \times B$ if and only if it holds for $A$ and $B$. 
\end{enumerate}
\end{corollary}
\begin{proof}
The first statement follows from Theorem \ref{maintheorem} and Equation (\ref{eq1}). The second statement follows from the fact that the minimal cohomology class of the product $A \times B$ is algebraic if and only if the minimal cohomology classes of the factors $A$ and $B$ are both algebraic.
\end{proof}

\begin{proof}[Proof of Theorem \ref{introth:IHCforjacobians}]
By Corollary \ref{cor:IHCforproducts} we may assume $n = 1$, so let $C$ be a smooth projective curve. Let $p \in C$ and consider the morphism $\iota\colon C \to J(C)$ defined by sending a point $q$ to the isomorphism class of the degree zero line bundle $\OO(p-q)$. 
Then $cl(\iota(C)) = \gamma_\theta \in \rm H^{2g-2}(J(C), \ZZ)$ by Poincar\'e's formula \cite{arbarello1985geometry}, so $\gamma_\theta$ is algebraic and the result follows from Theorem \ref{maintheorem}. 
\end{proof}

\begin{remarks} \label{symmetricpowerremark}
\begin{enumerate}[wide, labelwidth=!, labelindent=0pt]
    \item \label{symmetricremarkone} Let us give another proof of Theorem \ref{introth:IHCforjacobians} in the case $n = 1$, i.e.\ let $C$ be a smooth projective curve of genus $g$ and let us prove the integral Hodge conjecture for one-cycles on $J(C)$ in a way that does not use Fourier transforms. It is classical that any Abel-Jacobi map $C^{(g)} \to J(C)$ is birational. On the other hand, the integral Hodge conjecture for one-cycles is a birational invariant, see \cite[Lemma 15]{voisin_someaspectsofthehodgeconjecture}. Therefore, to prove it for $J(C)$ it suffices to prove it for $C^{(g)}$. One then uses \cite[Corollary 5]{delbanohodgecurve} which says that for each $n \in \ZZ_{\geq 1}$, there is a natural polarization $\eta$ on the $n$-fold symmetric product $C^{(n)}$ such that for any $i \in \ZZ_{\geq 0}$, the map $\eta^{n-i}\cup (-) \colon \rm H^i(C^{(n)}, \ZZ) \to \rm H^{2n-i}(C^{(n)}, \ZZ)$ is an isomorphism. In particular, the variety $C^{(n)}$ satisfies the integral Hodge conjecture for one-cycles for any positive integer $n$. 
    \item \label{symmetricremarktwo} Along these lines, observe that the integral Hodge conjecture for one-cycles holds not only for symmetric products of smooth projective complex curves but also for any product $C_1 \times \cdots \times C_n$ of smooth projective curves $C_i$ over $\CC$. Indeed, this follows readily from the K\"unneth formula. 
    \item Let $C$ be a smooth projective complex curve of genus $g$. Our proof of Theorem \ref{maintheorem} provides an explicit description of $\Hdg^{2g-2}(J(C),\ZZ)$ depending on $\Hdg^2(J(C),\ZZ)$. More generally, let $(A,\theta)$ be a principally polarized abelian variety of dimension $g$, and identify $A$ and $\wh A$ via the polarization. 
    Then $c_1(\ca P_A) = m^\ast(\theta) - \pi_1^\ast(\theta) - \pi_2^\ast(\theta)$, which implies that 
    \begin{align*}
\frac{1}{(2g-2)!} \cdot c_1(\ca P_A)^{2g-2} &= \sum_{\substack{i,j,k \geq 0\\ i + j + k = 2g-2}}^{2g-2}
(-1)^{j+k}\cdot  m^\ast\left( \frac{\theta^i}{i!}\right)\cdot \pi_1^\ast\left(\frac{\theta^j}{j!}\right)\cdot \pi_2^\ast \left(\frac{\theta^{k}}{k!}\right). 
    \end{align*}
    On the other hand, any $\beta \in \Hdg^{2g-2}(A,\ZZ)$ is of the form $$\beta = \pi_{2,\ast} \left( \frac{c_1(\ca P_A)^{2g-2}}{(2g-2)!}  \cdot \pi_1^\ast [D] \right) \in \Hdg^{2g-2}(A,\ZZ),$$ where we write $[D] = cl(D)$ for a divisor $D$ on $A$, as follows from \eqref{diagramwithonlysurfaceclasses}. Therefore, any element $\beta \in \Hdg^{2g-2}(A,\ZZ)$ may be written as
        \begin{equation} \label{beta}
\beta = \sum_{\substack{i,j,k \geq 0\\ i + j + k = 2g-2}}^{2g-2}
(-1)^{j+k}\cdot \pi_{2,\ast}\left( m^\ast\left( \frac{\theta^i}{i!}\right)\cdot \pi_1^\ast\left(\frac{\theta^j}{j!}\right)\cdot \pi_1^\ast[D]\right) \cdot \frac{\theta^{k}}{k!}. 
    \end{equation}

Returning to the case of a Jacobian $J(C)$ of a smooth projective curve $C$ of genus $g$, the classes $\theta^i/i!$ appearing in (\ref{beta}) are effective algebraic cycle classes. Indeed, for $p \in C$ and $d \in \ZZ_{\geq 1}$, the image of the morphism $C^{d} \to J(C)$, $(x_i) \mapsto \ca O(\sum_ix_i-d\cdot p)$ defines a subvariety $W_d(C) \subset J(C)$ and by Poincar\'e's formula \cite[\S I.5]{arbarello1985geometry} one has $cl(W_d(C)) = \theta^{g-d}/(g-d)! \in \rm H^{2g-2d}(J(C), \ZZ)$.

\end{enumerate}
\end{remarks}

Besides Theorem \ref{introth:IHCforjacobians}, we obtain the following corollary of Theorem \ref{maintheorem}:

\begin{corollary} \label{complexspecialization}
Let $A\to S$ be a principally polarized abelian scheme over a proper, smooth and connected variety $S$ over $\CC$. 
Let $X \subset S(\CC)$ be the set of $x \in S(\CC)$ such that the abelian variety $A_x$ satisfies the integral Hodge conjecture for one-cycles. Then $X = \cup_iZ_i(\CC)$ for some countable union of closed algebraic subvarieties $Z_i \subset S$. 
In particular, if the integral Hodge conjecture for one-cycles holds on $U(\CC)$ for a non-empty open subscheme $U$ of $S$, then it holds on all of $S(\CC)$. 
\end{corollary}
\begin{proof}
Write $\ca A = A(\CC)$ and $B = S(\CC)$ and let $\pi\colon \ca A \to B$ be the induced family of complex abelian varieties. Let $g \in \ZZ_{\geq 0}$ be the relative dimension of $\pi$ and define, for $t \in S(\CC)$, $\theta_t \in \NS(\ca A_t)\subset \rm H^2(\ca A_t, \ZZ)$ to be the polarization of $\ca A_t$. There is a global section $\gamma_\theta \in \rm R^{2g-2}\pi_\ast\ZZ$ such that for each $t \in B$, $\gamma_{\theta_t} = \theta_t^{g-1}/(g-1)! \in \rm H^{2g-2}(\ca A_t, \ZZ)$. 
Note that $\gamma_\theta$ is Hodge everywhere on $B$. For those $t \in B$ for which $\gamma_{\theta_t}$ is algebraic, write $\gamma_{\theta_t}$ as the difference of effective algebraic cycle classes on $\ca A_t$. This gives a countable disjoint union $\phi \colon \sqcup_{ij}H_i\times_S H_j \to S$ of products of relative Hilbert schemes $H_i \to S$. By Lemma \ref{lemma_algebraicityminimalclass} below, $\gamma_{\theta_t}$ is algebraic precisely for closed points $t$ in the image $Y \subset S$ of $\phi$. Theorem \ref{maintheorem} implies that $X = Y$ and the assertion is proven. 
\end{proof}

\begin{lemma} \label{lemma_algebraicityminimalclass}
Let $S$ be an integral variety over $\CC$, let $\ca A \to S$ be a principally polarized abelian scheme of relative dimension $g$ over $S$ and let $\ca C_i \subset \ca A$ for $i= 1,\dotsc, k$ be relative curves in $\ca A$ over $S$. Let $n_1,\dotsc, n_k$ be integers and let $y \in S(\CC)$ be a point that satisfies $\sum_{i = 1}^k n_i\cdot cl(C_{i,y}) = \gamma_{\theta_y} \in \rm H^{2g-2}(A_y,\ZZ)$. Then for every $x \in S(\CC)$, one has $\sum_{i = 1}^k n_i\cdot cl(C_{i,x}) = \gamma_{\theta_x} \in \rm H^{2g-2}(A_x,\ZZ)$. 
\end{lemma}

\begin{proof}
Since it suffices to prove the lemma for any open affine $U \subset S$ that contains $y$, we may assume that $S$ is quasi-projective. Fix $x \in S(\CC)$. After replacing $S$ by a suitable base change containing $x$ and $y$, we may assume that $S$ is an open subscheme of a smooth connected curve. 
For $t \in S$, denote by $\theta_{\bar t} \in \rm H^{2}_{\etale}(A_{\bar t}, \ZZ_\ell)$ the class of the polarization and $\gamma_{\theta_{\bar t}} = \theta_{\bar t}^{g-1}/(g-1)!$. 
Let $\eta = \Spec K$ be the generic point of $S$. The elements $\sum_i n_i \cdot cl(C_{i,\bar \eta})$ and $\gamma_{\theta_{\bar \eta}}$ in $\rm H^{2g-2}_{\etale}(A_{\bar \eta}, \ZZ_\ell)$ both map to $\sum_i n_i \cdot cl(C_{i,y}) = \gamma_{\theta_{y}} \in  \rm H^{2g-2}_{\etale}(A_{y}, \ZZ_\ell)$ under the specialization homomorphism $s\colon \rm H^{2g-2}_{\etale}(A_{\bar \eta}, \ZZ_\ell) \to \rm H^{2g-2}_{\etale}(A_{y}, \ZZ_\ell)$ 
by \cite[Example 20.3.5]{fultonintersection}. Since $s$ is an isomorphism, we have $\sum_i n_i \cdot cl(C_{i,\bar \eta}) = \gamma_{\theta_{\bar \eta}}$, which implies that $\sum_{i}n_i\cdot cl(C_{x,i}) = \gamma_{\theta_x} \in \rm H_{\etale}^{2g-2}(A_x,\ZZ_\ell)$.
\end{proof}

\subsection{Density of abelian varieties satisfying IHC$_1$} \label{subsec:density}

The goal of this section is to prove that Conditions $\ref{introitem:minimalpoincare}-\ref{introitem:integralhodgeforproduct}$ in Theorem \ref{maintheorem} are satisfied on a dense subset of the moduli space of complex abelian varieties. To do so, will we state yet another criterion that a complex abelian variety may satisfy. In some sense this criterion provides a bridge between abelian varieties outside the Torelli locus and those lying within, thereby implying the integral Hodge conjecture for one-cycles for the abelian variety under consideration. 

\begin{definition} \label{definition:primeisogenies}
Let $A$ and $B$ be a complex abelian varieties and let $p$ a prime number. We say that $A$ is \textit{prime-to-$p$ isogenous to a $B$} if there exists an isogeny $\alpha \colon A \to B$ whose degree $\deg(\alpha)$ is not divisible by $p$. We say that $A$ is \textit{$p$-power isogenous to $B$} if $A$ is isogenous to $B$ for some isogeny $\alpha$ whose degree is a power of $p$. 
\end{definition}

The following proposition shows in particular that to prove the density part of the statement in Theorem~\ref{introth:density}, it suffices to prove that for any prime number $\ell$, those abelian varieties that are $\ell$-power isogenous to a product of elliptic curves are dense in their moduli space.

\begin{proposition} \label{proposition:IHCONELEMMA}
Let $A$ be a complex abelian variety of dimension $g$. Let $\wh A$ be the dual abelian variety and let $\ca P_A$ be the Poincar\'e bundle. Let $\kappa$ be a non-zero integer such that the cohomology class $\kappa \cdot c_1(\ca P_A)/(2g-1)! \in \rm H^{4g-2}(A \times \wh A, \ZZ)$ is algebraic. Consider the following statements:
\begin{enumerate}[wide, labelwidth=!, labelindent=0pt]
    \item \label{itemlemma:one} The abelian variety $A$ satisfies the integral Hodge conjecture for one-cycles.
    \item \label{itemlemma:two} For every prime number $p$, there exists an abelian variety $B$ such that the abelian variety $A \times B$ is prime-to-$p$ isogenous to the Jacobian of a smooth projective curve. 
    \item \label{itemlemma:twosmall} For every prime number $p$ that divides $\kappa$, there exists an abelian variety $B$ such that the abelian variety $A \times B$ is prime-to-$p$ isogenous to a Jacobian of a smooth projective curve. 
    \item \label{itemlemma:three} For every prime number $p$, there exists an abelian variety $B$ such that the abelian variety $A \times B$ is prime-to-$p$ isogenous to a product of Jacobians of smooth projective curves.
     \item \label{itemlemma:threesmall} For every prime number $p$ dividing $\kappa$, there exists an abelian variety $B$ such that the abelian variety $A \times B$ is prime-to-$p$ isogenous to a product of Jacobians of smooth projective curves. 
\end{enumerate}
Then $[\ref{itemlemma:two} \implies \ref{itemlemma:twosmall} \implies \ref{itemlemma:threesmall} \implies \ref{itemlemma:one}]$ and $[\ref{itemlemma:two} \implies \ref{itemlemma:three} \implies \ref{itemlemma:threesmall}]$. Moreover, if $A$ is principally polarized by $\theta_A \in \NS(A)$, then \ref{itemlemma:one}  is implied by 
\begin{enumerate}[wide, labelwidth=!, labelindent=0pt]
\setcounter{enumi}{5}
    \item \label{itemlemma:zero} For any prime number $p | (g-1)!$ there exists a smooth projective curve $C$ and a morphism of abelian varieties $\phi\colon A \to J(C)$ such that $\phi^\ast \theta_{J(C)} = m\cdot \theta_A$ for $m \in \ZZ_{\geq 1}$ with $\gcd(m,p) = 1$. 
\end{enumerate}
Finally, if $A$ is principally polarized of Picard rank one, then the statements $\ref{itemlemma:one} - \ref{itemlemma:zero} $ are equivalent.  
\end{proposition}

\begin{proof}
\textbf{Step one}: $[\ref{itemlemma:two} \implies \ref{itemlemma:twosmall} \implies \ref{itemlemma:threesmall}]$ \textit{and} $[\ref{itemlemma:two} \implies \ref{itemlemma:three} \implies \ref{itemlemma:threesmall}]$. These implications are trivial. 
\\
\\
\textbf{Step two}: $[\ref{itemlemma:threesmall} \implies \ref{itemlemma:one}]$. 
Let $g$ be the dimension of $A$. We want to prove that the class $c_1(\ca P_A)^{2g-1}/(2g-1)! \in \rm H^{4g-2}(A \times \wh A, \ZZ)$ is algebraic. Let $p$ be any prime number that divides $\kappa$. Then by Condition \ref{itemlemma:threesmall}, there exists an abelian variety $B$ and an isogeny $\alpha\colon A \times B \to Y$ to the product $Y = \prod_iJ(C_i)$ of Jacobians $J(C_i)$ of smooth projective curves $C_i$ such that $\gcd(\deg(\alpha), p) = 1$. Define $X = A \times B$. Let $g_B$ be the dimension of $B$, let $h = g + g_B = \dim(X) = \dim(Y)$, and let $m_p = \deg(\alpha)$. There exists an isogeny $\beta \colon Y \to X$ such that 
$\beta \circ \alpha = [m_p]_X$. If we define $n_p = \deg(\beta)$ then $m_p \cdot n_p = \deg(\alpha) \cdot \deg(\beta) = \deg(\alpha \circ \beta) = m_p^{2h}$. 
Therefore, $(\beta \circ \alpha) \times (\wh \alpha \circ \wh \beta) = [m_p]_{X \times \wh X}$. 
Consequently, if $N_p = 2h \cdot (4h-2)$, then the homomorphism $$[m_p^{2h}]^{\ast} = (m_p^{N_p} \cdot (-) ) \colon \rm H^{4h-2}(X \times \wh X, \ZZ) \to \rm H^{4h-2}(X \times \wh X, \ZZ)$$ 
factors through $\rm H^{4h-2}(Y \times \wh{Y}, \ZZ)$. Since $Y \times \wh{Y}$ satisfies the integral Hodge conjecture by Theorem \ref{introth:IHCforjacobians}, the Hodge class $m_p^{N_p} \cdot c_1(P_X)^{2h-1}/(2h-1)! \in \rm H^{4h-2}(X \times \wh X , \ZZ)$ is algebraic. 
Let $f\colon A \times B \times \wh A \times \wh B \to A \times \wh A$ and $g \colon A \times B \times \wh A \times \wh B \to B \times \wh B$ be the canonical projections. Then $\ca P_X \cong f^\ast \ca P_A \otimes g^\ast \ca P_B$. Using this and denoting $\mu = c_1(\ca P_A)$ and $\nu = c_1(\ca P_B)$ we have
\[
    \frac{c_1(\ca P_{X})^{2h-1}}{(2h-1)!} =f^{\ast}\left(\frac{\mu^{2g-1}}{(2g-1)!}\right) \cdot g^\ast \left( \frac{\nu^{2g_B}}{(2g_B)!} \right) + f^\ast \left( \frac{\mu^{2g}}{(2g)!} \right) \cdot  g^\ast\left(\frac{\nu^{2g_B-1}}{(2g_B-1)!}    \right).\] 
    This implies that 
$f_\ast \left(c_1(\ca P_X)^{2h-1}/(2h-1)! \right) = (-1)^{g_b}\mu^{2g-1}/(2g-1)!$.
In particular, the class $m_p^{N_p} \cdot c_1(\ca P_A)^{2g-1}/(2g-1)! \in \rm H^{4g-2}(A \times \wh A, \ZZ)$ is algebraic. 

Let $p_1, \dotsc, p_n$ be all prime divisors of $\kappa$ and observe that $\gcd(\kappa, m_{p_1}^{N_{p_1}},m_{p_2}^{N_{p_2}}, \dotsc, m_{p_n}^{N_{p_n}}) = 1$. 
Therefore, there are integers $a, b_1, \dotsc, b_n$ such that $a \cdot \kappa + \sum_{i = 1}^n b_i \cdot m_{p_i}^{N_{p_i}} = 1$. One obtains 
\[
\frac{c_1(\ca P_A)^{2g-1}}{(2g-1)!} =
a \cdot \kappa \cdot \frac{c_1(\ca P_A)^{2g-1}}{(2g-1)!}  + \sum_{i = 1}^n b_i \cdot m_{p_i}^{N_{p_i}} \cdot \frac{c_1(\ca P_A)^{2g-1}}{(2g-1)!} \in \rm H^{4g-2}(A \times \wh A, \ZZ).
\]
This proves that $c_1(\ca P_A)^{2g-1}/(2g-1)!$ is a $\ZZ$-linear combination of algebraic classes, hence algebraic. Condition \ref{itemlemma:one} follows then from Theorem \ref{maintheorem}.
\\
\\
\textbf{Step three}: [$\ref{itemlemma:zero} \implies \ref{itemlemma:one}]$ \textit{for $A$ principally polarized by $\theta_A \in \NS(A)$}. Let $p_1,\dotsc, p_k$ be the prime factors of $(g-1)!$ and let $C_1,\dotsc, C_k$ be smooth proper curves for which there exist homomorphisms $\phi_i \colon A \to J(C_i)$ such that $\phi^\ast \theta_{J(C_i)} = m_i \cdot \theta_A$ for some $m_i \in \ZZ_{\geq 1}$ with $p_i \nmid m$. Since $\theta_{J(C_i)}^{g-1}/(g-1)! \in \rm H^{2g-2}(J(C_i),\ZZ)$ is algebraic for each $i$, the classes $\phi_i^\ast(\theta_{J(C_i)}^{g-1}/(g-1)!) = m_i^{g-1} \cdot \theta_{A}^{g-1}/(g-1)! \in \rm H^{2g-2}(A,\ZZ)$ are algebraic. Since $\gcd((g-1)!, m_1,\dotsc, m_k) = 1$, this implies that $\theta_A^{g-1}/(g-1)!$ is algebraic. Condition \ref{itemlemma:one} follows then from Theorem \ref{maintheorem}. 
\\
\\
\textbf{Step four}: [$\ref{itemlemma:zero} \impliedby \ref{itemlemma:one} \implies \ref{itemlemma:two}]$ \textit{for $(A,\theta_A)$ principally polarized with $\rho(A) =1$}. 
Write $\theta = \theta_A$. Let $Z_1, \dotsc, Z_n$ be integral curves $Z_i \subset A$ and let $e_1, \dotsc, e_n \in \ZZ$ with $e_i \neq 0$ for all $i$ be such that 
$
{\theta^{g-1}}/{(g - 1)!} = \sum_{i = 1}^n e_i\cdot [Z_i] \in \rm H^{2g - 2}(A, \ZZ)$. Since $\rho(A) = 1$, the group $\Hdg^{2g-2}(A, \ZZ)$ is generated by $\theta^{g-1}/(g-1)!$. Consequently, we have $[Z_i] = f_i\cdot \left(\theta^{g-1}/(g-1)!\right)$ for some non-zero $f_i \in \ZZ$. Hence we can write 
$$
{\theta^{g-1}}/{(g - 1)!} = \sum_{i = 1}^n e_i\cdot [Z_i]  = \sum_{i = 1}^n e_i\cdot f_i \cdot \theta^{g-1}/(g - 1)!$$ which implies that $\sum_{i = 1}^n e_i\cdot f_i  = 1$. Now let $p$ be any prime number. Then there exists an integer $i$ with $1 \leq i \leq n$ such that $p$ does not divide $f_i$. Let $C_i \to Z_i$ be the normalization of $Z_i$ and let $\lambda_A = \varphi_{\theta} \colon A \to \wh A$ be the polarization corresponding to $\theta$. This gives a diagram
\begin{equation}
\begin{split}
\label{diagram:embedintojacobinian}
\xymatrixcolsep{4pc}
\xymatrix{
C_i \ar[rr]^{\varphi} \ar[dr]^{\iota}&& A \ar@/_3pc/[rrr]^{\phi} \ar[r]_{\sim}^{\lambda_A} & \wh A \ar[r]^{\varphi^{\ast}} & \Pic^0(C_i) \ar[r]^{a}_{\sim} & J(C_i), \\
& J(C_i) \ar[ur]^{\psi} & & &  &
}
\end{split}
\end{equation}
where $\iota\colon C_i \to J(C_i) = H^0(C, \Omega_C)^\ast/H_1(C,\ZZ)$ is the Abel--Jacobi map (for some $p \in C$), and $\varphi^\ast\colon \wh A = \Pic^0(A) \to \Pic^0(C_i)$ is the pullback of line bundles along $\varphi\colon C_i \to A$. The natural homomorphism $a\colon \Pic^0(C_i) \to J(C_i)$ is an isomorphism by the Abel--Jacobi theorem. Since the triangle on the left in Diagram (\ref{diagram:embedintojacobinian}) commutes and $[Z_i] \in \rm H^{2g-2}(A, \ZZ)$ is non-zero, the morphism $\psi \colon J(C_i) \to A$ is non-zero. As $\rho(A) = 1$, the map $\psi\colon J(C_i) \to A$ must be surjective, the Picard rank of a non-simple abelian variety being greater than one. Dually, $\psi$ gives rise to a non-zero homomorphism $\wh \psi\colon \wh A \to \wh{J(C_i)}$, and the simpleness of $\wh A$ implies that $\wh \psi$ is finite onto its image. We claim that the same is true for $\phi$. To prove this, it suffices to show that the kernel of $\varphi^\ast \colon \wh A \to \Pic^0(C_i)$ is finite. Since the homomorphism $\iota^\ast\colon \wh{J(C_i)} \to \Pic^0(C_i)$ induced by the embedding $\iota\colon C_i \to J(C_i)$ is an isomorphism, dualizing the triangle on the left in Diagram (\ref{diagram:embedintojacobinian}) proves our claim. 
By construction, we have 
$\varphi_{\ast}[C_i] = [Z_i] = f_i \cdot \theta^{g-1}/(g-1)! \in \rm H^{2g-2}(A, \ZZ)$. By a version of Welters' Criterion (see \cite[Lemma 12.2.3]{birkenhake}), this implies that $\phi^{\ast}\left(\theta_{J(C_i)}\right) = f_i \cdot \theta \in \rm H^2(A, \ZZ)$, where $\theta_{J(C_i)} \in \rm H^2(J(C_i), \ZZ)$ is the canonical principal polarization. In particular, \ref{itemlemma:zero} holds. 

We claim that also \ref{itemlemma:two} holds. Let $j\colon A_0 \hookrightarrow J(C_i)$ be the embedding of $A_0 = \phi(A)$ into $J(C_i)$ and let $\lambda_0 \colon A_0 \to \wh A_0$ be the polarization on $A_0$ induced by $j$. We have $\phi^{\ast}(\lambda) = \varphi_{f_i\cdot \theta} = f_i \cdot \varphi_\theta = f_i \cdot \lambda_A$. We obtain a commutative diagram
\[
\xymatrixcolsep{5pc}
\xymatrix{
&A\ar[dl]_{[f_i]_A} \ar[r]^{\pi}\ar[d]^{f_i\cdot \lambda_A} &A_0\ar[d]^{\lambda_0} \ar[r]^j & J(C_i) \ar[d]^{\lambda} \\
A & \wh A \ar[l]^{\lambda_{\wh A}} & \wh A_0 \ar[l]^{\wh \pi} & \widehat{J(C_i)}. \ar[l]\ar[l]^{\wh j} 
}
\]
Let $G$ be the kernel of $\pi$. Define $K = \Ker([f_i]_A) = \Ker(f_i\cdot \lambda_A) \cong (\ZZ/f_i)^{2g} \subset A$, and $U = \Ker(\wh \pi \circ \lambda_0) \subset A_0$. Also define $H = \Ker(\lambda_0)$, and observe that $H \subset U$. The exact sequence $0 \to G \to K \to U \to 0$ shows that if $a, k, u$ and $h$ are the respective orders of $G$, $K$, $U$ and $H$, then one has
\begin{equation} \label{eq:dividingorders}
h | u | k | f_i \quad \textnormal{ and } \quad a | k | f_i.
\end{equation}
Then define $B = \Ker( \wh j \circ \lambda ) \subset J(C_i)$ with inclusion $i\colon B \hookrightarrow J(C_i)$. It is easy to see that $B$ is connected. 
Moreover, we have $A_0 \cap B = H$
and, therefore, an exact sequence of commutative group schemes $$0 \to H \to A_0 \times B \xrightarrow{\psi} J(C_i) \to 0.$$ 
The morphism $\alpha\colon A \times B \to J(C_i)$, defined as the composition
$
\xymatrixcolsep{2pc}
\xymatrix{
A \times B \ar[r]^{\pi \times \id}  & A_0 \times B \ar[r]^{\psi} & J(C_i),}$ is an isogeny. Since the degree of an isogeny is multiplicative in compositions, we have $\deg(\alpha) = \deg\left(\psi \circ (\pi \times \id) \right) =  \deg(\psi) \cdot \deg(\pi \times \id)  = h \cdot \deg(\pi) = h \cdot a$. In particular, $p$ does not divide $\deg(\alpha)$ because $h$ and $a$ divide $f_i$ by Equation (\ref{eq:dividingorders}).
\end{proof}

\begin{proof}[Proof of Theorem \ref{introth:density}]
According to Theorem \ref{maintheorem}, it suffices to show that the cohomology class $c_1(\ca P_A)^{2g-1}/(2g-1)! \in \rm H^{4g-2}(A \times \wh A, \ZZ)$ is algebraic for $[(A,\lambda)]$ in a dense subset $X$ of $\msf A_{g,\delta}(\CC)$ as in the statement. Define $D = \textnormal{diag}(\delta_1, \dotsc, \delta_g)$ and define, for each subring $R$ of $\CC$, a group
\[
\textnormal{Sp}_{2g}^\delta(R) = \left\{M \in \GL_{2g}(R) \mid  M 
\begin{pmatrix} 
0 & D \\
-D & 0
\end{pmatrix} M^t = 
\begin{pmatrix} 
0 & D \\
-D & 0
\end{pmatrix}
\right\}. 
\]
The isomorphism
\[
\textnormal{Sp}_{2g}^\delta(\RR)  \to \textnormal{Sp}_{2g}(\RR), \quad M \mapsto 
\begin{pmatrix} 
1_g & 0 \\
0 & D
\end{pmatrix}^{-1} M \begin{pmatrix} 
1_g & 0 \\
0 & D
\end{pmatrix}
\]
induces an action of $\textnormal{Sp}_{2g}^\delta(\ZZ)$ on the genus $g$ Siegel space $\bb H_g$, and the period map defines an isomorphism of complex analytic spaces $\msf A_{g,\delta}(\CC) \cong \textnormal{Sp}_{2g}^\delta(\ZZ) \setminus \bb H_g$ \cite[Theorem 8.2.6]{birkenhake}. Pick any prime number $\ell > (2g-1)!$ and consider, for a period matrix $x \in \bb H_g$, the orbit $\textnormal{Sp}_{2g}^\delta(\ZZ[1/\ell]) \cdot x \subset \bb H_g$. Let $(A, \lambda)$ be a polarized abelian variety admitting a period matrix equal to $x$. The image of $\textnormal{Sp}_{2g}^\delta(\ZZ[1/\ell]) \cdot x$ in $\msf A_{g,\delta}(\CC)$ is the \textit{Hecke-$\ell$-orbit} of $[(A, \lambda)] \in \msf A_{g,\delta}(\CC)$, i.e.\ the set of isomorphism classes of polarized abelian varieties $[(B, \mu)] \in \msf A_{g,\delta}(\CC)$ for which there exists integers $n,m\in \ZZ_{\geq 0}$ and an isomorphism of polarized rational Hodge structures $\phi\colon \rm H_1(B, \QQ) \xrightarrow{\sim} \rm H_1(A, \QQ)$ 
such that $\ell^n \cdot \phi$ and $\ell^m \cdot \phi^{-1}$ are morphisms of integral Hodge structures (Hecke orbits were studied in positive characteristic in e.g.\ \cite{chaiordinaryhecke, oorthecke}). The degree of the isogeny $\alpha = \ell^n\phi$ must be $\ell^k$ for some nonnegative integer $k$. 
In particular, if one abelian variety in a Hecke-$\ell$-orbit happens to be isomorphic to a Jacobian, then every abelian variety in that orbit is $\ell$-power isogenous to a Jacobian, see Definition \ref{definition:primeisogenies}. 

The decomposition of a polarized abelian variety into non-decomposable polarized abelian subvarieties is unique \cite[Corollaire 2]{debarreproduits}, which implies that the following morphism
\[
\pi \colon \prod_{i = 1}^g \msf A_{1,1} \to \msf A_{g,\delta}, \quad
\left([(E_1, \lambda_1)], \dotsc, [(E_g, \lambda_g)] \right) \mapsto ([E_1 \times \cdots \times E_g, \delta_1\cdot \lambda_1 \times \cdots \times \delta_g\cdot \lambda_g)]
\]
is finite onto its image. Thus $\msf A_{g, \delta}$ contains a $g$-dimensional subvariety on which the integral Hodge conjecture for one-cycles holds. We claim that $\Sp_{2g}^\delta(\ZZ[1/\ell])$ is dense in $\Sp_{2g}(\RR)$. Since $\Sp_{2g}^{\delta}(\QQ)$ arises as the group of rational points of an algebraic subgroup $\Sp_{2g}^\delta$ of $\GL_{2g}$ over $\mathbb Q$ \cite[Chapter 2, \S 2.3.2]{PlatonovRapinchuk}, which is isomorphic to $\Sp_{2g}$ over $\QQ$, 
this claim follows from the well-known fact that for $S = \{\ell\} \subset \Spec \ZZ$, the algebraic group $\Sp_{2g}$ satisfies the strong approximation property with respect to $S$ \cite[Chapter 7, \S 7.1]{PlatonovRapinchuk} (indeed, this is classical and follows from the non-compactness of $\Sp_{2g}(\QQ_\ell)$, see \cite[Theorem 7.12]{PlatonovRapinchuk}). 

Let $V = \pi\left( \prod_{i = 1}^g \msf A_{1,1}\right) \subset \msf A_{g,\delta}$. Then $X' \coloneqq \Sp_{2g}^\delta(\ZZ[1/\ell]) \cdot V =  \cup_iZ_i \subset \msf A_{g, \delta}(\CC)$ is a countable union of closed analytic subsets $Z_i \subset \msf A_{g, \delta}(\CC)$ of dimension $\dim Z_i \geq g$ such that $X' \subset \msf A_{g, \delta}(\CC)$ is dense in the analytic topology and $c_1(\ca P_A)^{2g-1}/(2g-1)! \in \rm H^{4g-2}( A \times \wh A, \ZZ)$ is algebraic for every polarized abelian variety $(A, \lambda)$ of polarization type $\delta$ whose isomorphism class lies in $X'$. To prove the theorem, we are reduced to proving that there exists a similar countable union $X \subset \msf A_{g, \delta}(\CC)$ whose components are algebraic. For this, it suffices to prove the following \textit{claim:} the locus of $[(A, \lambda)] \in \msf A_{g, \delta}(\CC)$ such that $c_1(\ca P_A)^{2g-1}/(2g-1)! \in \rm H^{4g-2}( A \times \wh A, \ZZ)_{\textnormal{alg}}$ is a countable union $W = \cup_jY_j \subset \msf A_{g, \delta}(\CC)$ of closed algebraic subsets $Y_j \subset \msf A_{g, \delta}(\CC)$. Indeed, if this holds, then $X' \subset W$ and 
since each $Z_i \subset X$ is irreducible, each $Z_i$ is contained in an irreducible component $Y_j \subset W$. We may then define $X$ as the union of those $Y_j \subset W$ that contain some $Z_i$. 

To prove the claim, let $U \to \ca A_{g, \delta}$ be a finite \'etale cover of the moduli stack $\ca A_{g, \delta}$ and let $\ca X \to U$ be the pullback of the universal family of abelian varieties along $U \to \ca A_{g, \delta}$. This gives an abelian scheme $\ca X \times \wh{\ca X} \to U$ carrying a relative Poincar\'e line bundle $\ca P_{\ca X/U}$ and arguments similar to those used to prove Lemma \ref{lemma_algebraicityminimalclass} show that indeed, for each irreducible component $U' \subset U$, the locus in $U'(\CC)$ where $c_1(\ca P_A)^{2g-1}/(2g-1)!$ is algebraic is a countable union of closed algebraic subvarieties of $U'(\CC)$. 

Finally, Theorem \ref{maintheorem} implies that for each $[(A, \lambda)] \in X$, the integral Hodge conjecture for one-cycles holds for the abelian variety $A$, so we are done. 
\end{proof}

\begin{remark} \label{rem:dimensionimprovement}
Using level structures one can show that whenever $\gcd(\prod_i\delta_i, (2g-1)!) = 1$ (or, more generally, $\gcd(\prod_i\delta_i, (2g-2)!) = 1$, see Section \ref{sec:integralhodgeuptofactorn} below), there is a countable union $X = \cup_iZ_i \subset \msf A_{g, \delta}(\CC)$ as in Theorem \ref{introth:density} such that $\dim Z_i \geq 3g-3$. Indeed, let $\msf A_{g, \delta_g}^\ast$ be the moduli space of principally polarized abelian varieties of dimension $g$ with $\delta_g$-level structure. Then there is a natural morphism $\phi \colon \msf A_{g, \delta_g}^\ast \to \msf A_{g, \delta}$ such that for any $ x= [(A, \lambda)] \in \msf A_{g, \delta_g}^\ast(\CC)$ with $[(B, \mu)] = \phi(x) \in \msf A_{g, \delta}(\CC)$, there exists an isogeny $\alpha\colon A \to B$ of degree $\prod_{i = 1}^g \delta_i$, see \cite{mumfordmoduliofabelianvarieties}.
\end{remark}

\begin{remark}
In the principally polarized case, the density in the moduli space of those abelian varieties that satisfy the integral Hodge conjecture for one-cycles admits another proof which might be interesting for comparison. Let $\msf A_g$ be the coarse moduli space of principally polarized complex abelian varieties of dimension $g$ and let $[(A,\theta)]$ be a closed point of $\msf A_g$. Then by \cite[Exercise 5.6.(10)]{birkenhake}, the following are equivalent: (i) $A$ is isogenous to the $g$-fold self-product $E^g$ for an elliptic curve $E$ with complex multiplication, (ii) $A$ has maximal Picard rank $\rho(A) = g^2$, (iii) $A$ is \textit{isomorphic} to the product $E_1 \times \cdots \times E_g$ of pairwise isogenous elliptic curves $E_i$ with complex multiplication. If any of these conditions is satisfied, then $A$ satisfies the integral Hodge conjecture for one-cycles by Theorem \ref{introth:IHCforjacobians}. Moreover, the set of isomorphism classes of principally polarized abelian varieties $(A,\theta)$ for which this holds is dense in $\msf A_g$ by \cite{lange_modulprobleme}. For an explicit example in dimension $g = 4$ of a principally polarized abelian variety $(A, \theta)$ that satisfies one of the equivalent conditions above, but which is not isomorphic to a Jacobian, see \cite[\S 5]{debarre_annulation}. 
\end{remark}


\section{The integral Hodge conjecture for one-cycles up to factor $n$}\label{sec:integralhodgeuptofactorn}

In this section, we study a property of a smooth projective complex variety that lies somewhere in between the integral Hodge conjecture and the usual (i.e.\ rational) Hodge conjecture. The key will be the following:

\begin{definition}
Let $d,k,n \in \ZZ_{\geq 1}$ and let $X$ be a smooth projective variety over $\CC$ of dimension $d$. Recall the definition of the degree $2d-2k$ \textit{Voisin group} of $X$ \cite{voisinstablyrational, perry2020integral}:
    \[
    \rm Z^{2d-2k}(X) \coloneqq \textnormal{Hdg}^{2d-2k}(X, \ZZ)/ \rm H^{2d-2k}(X,\ZZ)_{\textnormal{alg}} = \Coker \left( \CH_k(X) \to \textnormal{Hdg}^{2d-2k}(X, \ZZ) \right).
    \]
 We say that \textit{$X$ satisfies the integral Hodge conjecture for $k$-cycles up to factor $n$} if $\rm Z^{2d-2k}(X)$ is annihilated by $n$ (in other words, if $n \cdot x \in \rm H^{2d-2k}(X,\ZZ)_{\textnormal{alg}}$ for every $x \in \textnormal{Hdg}^{2d-2k}(X, \ZZ)$). 
\end{definition}



\begin{lemma} \label{lemma:integralhodgeupton}
Let $A$ be a complex abelian variety of dimension $g$. 
\begin{enumerate}[wide, labelwidth=!, labelindent=0pt]
    \item  \label{firstiem}
Let $n$ be a positive integer and let $\ca F_n\colon \CH^1(\wh A) \to \CH_1(A)$ be a group homomorphism such that the following diagram commutes:
\[
\xymatrixcolsep{5pc}
\xymatrix{
\rm \CH^1(\wh A) \ar[d]\ar[r]^{\ca F_n} & \CH_1(A)\ar[d] \\
\rm H^2(\wh A,\ZZ)\ar[r]^{n \cdot \mr F_{\wh A} } & \rm H^{2g-2}( A, \ZZ).
}
\]
Then $A$ satisfies the integral Hodge conjecture for one-cycles up to factor $n$. 
\item \label{secondiem} Let $n \in \ZZ_{\geq 1}$ be such that $n \cdot c_1(\ca P_A)^{2g-2}/(2g-2)!$ is algebraic. Then a homomorphism $\ca F_n$ 
as in \ref{firstiem} exists. 
\end{enumerate}
\end{lemma}
\begin{proof}
Statement \ref{firstiem} follows immediately from the fact that $\CH^1(\wh A) \to \Hdg^2(\wh A,\ZZ)$ is surjective by Lefschetz $(1,1)$. To prove \ref{secondiem}, define $\sigma_A \in \rm H^{4g-4}(A \times \wh A,\ZZ)$ to be the class $c_1(\ca P_A)^{2g-2}/(2g-2)!$. First observe that if $\sigma_{\wh A} \coloneqq c_1(\ca P_{\wh A})^{2g-2}/(2g-2)! \in \rm H^{4g-4}(\wh A \times A,\ZZ)$, then $n \cdot \sigma_{\wh A}$ is algebraic since $n \cdot \sigma_A$ is. Let $\Sigma_n \in \CH_2(\wh A \times A)$ be such that $cl(\Sigma_n) = n \cdot \sigma_{\wh A}$. This gives a commutative diagram:
\begin{equation*} 
\xymatrixcolsep{5pc}
\xymatrix{
\CH^1(\wh A) \ar[d]^{cl}\ar[r]^{\pi_1^\ast}& \CH^1(\wh A \times A) \ar[d]^{cl}\ar[r]^{\cdot \Sigma_n}& \CH^{2g-1}(\wh A \times A) \ar[d]^{cl} \ar[r]^{\pi_{2,\ast}} & \CH_1( A) \ar[d]^{cl} \\
\rm H^2(\wh A,\ZZ)
\ar[r]^{\pi_1^\ast} & \rm H^{2}(\wh A \times A,\ZZ) \ar[r]^{\cdot n \cdot \sigma_{\wh A}} & \rm H^{4g-2}(\wh A \times  A,\ZZ) \ar[r]^{\pi_{2,\ast}} & \rm H^{2g-2}( A,\ZZ). 
}
\end{equation*}

Since $ \pi_{2,\ast} \circ \left((-)\cdot n \cdot \sigma_{\wh A}\right) \circ \pi_1^\ast = n \cdot \mr F_{\wh A}$, the homomorphism $\ca F_n \coloneqq \pi_{2,\ast} \circ \left((-)\cdot \Sigma_n\right) \cdot \pi_1^\ast$ has the required property.  \end{proof}

\begin{theorem} \label{theorem:integralhodgeuptofactor}
Consider a complex abelian variety $A$ of dimension $g$. 
\begin{enumerate}[wide, labelwidth=!, labelindent=0pt]
    \item \label{item:nphodgealg} 
    Let $n \in \ZZ_{\geq 1}$ be such that 
    $n \cdot c_1(\ca P_A)^{2g-1}/(2g-1)!$ 
    is algebraic. Then $n^2 \cdot c_1(\ca P_A)^{2g-2}/(2g-2)!$ is algebraic. In particular, $A$ satisfies the integral Hodge conjecture up to factor $\gcd(n^2, (2g-2)!)$ in this case. 
    \item \label{item:ppavnphodgealg} If $A$ is principally polarized, and $ n \in \ZZ_{\geq 1}$ is such that $n \cdot \gamma_\theta \in \textnormal{Hdg}^{2g-2}(A, \ZZ)$ is algebraic, then $n \cdot c_1(\ca P_A)^{2g-1}/(2g-1)!\in \textnormal{Hdg}^{4g-2}(A \times \wh A, \ZZ)$ is algebraic. 
    \item \label{item:ppavnphodgealg-last}
    We have that $A$ satisfies the integral Hodge conjecture for one-cycles up to factor $(2g-2)!$, and Prym varieties satisfy the integral Hodge conjecture for one-cycles up to factor $4$. 
\end{enumerate}
\end{theorem}
\begin{proof}
\ref{item:nphodgealg}. 
By Lemma \ref{lemma:crucialprop}, one has \[
c_1(\ca P_A)^{2g-2}/(2g-2)! = (-1)^g\cdot \left( c_1(\ca P_A)^{2g-1}/(2g-1)!\right)^{\star 2}/2! \in \rm H^{4g-4}(A \times \wh A,\ZZ). 
\]
By Theorem \ref{th:PDstructure}, this implies that if $n \cdot c_1(\ca P_A)^{2g-1}/(2g-1)!$ is algebraic, then also the element $n^2\cdot c_1(\ca P_A)^{2g-2}/(2g-2)!$ is algebraic. Since $(2g-2)! \cdot c_1(\ca P_A)^{2g-2}/(2g-2)!$ is algebraic, it follows that $\gcd(n^2, (2g-2)!) \cdot c_1(\ca P_A)^{2g-2}/(2g-2)!$ is algebraic. Thus we are done by Lemma \ref{lemma:integralhodgeupton}.

\ref{item:ppavnphodgealg}. This follows from Lemma \ref{lemma:minimalclasspoincarecomparison}.

\ref{item:ppavnphodgealg-last}. This follows from Lemma \ref{lemma:integralhodgeupton}, parts \ref{item:nphodgealg} and \ref{item:ppavnphodgealg} and the fact that if $A$ is a $g$-dimensional Prym variety with principal polarization $\theta \in \Hdg^2(A,\ZZ)$, then $2 \cdot \gamma_\theta \in \rm H^{2g-2}(A,\ZZ)$ is algebraic.  
\end{proof}


\section{The integral Tate conjecture for one-cycles on abelian varieties over the separable closure of a finitely generated field}

Let $X$ be a smooth projective variety over the separable closure $k$ of a finitely generated field. Let $k_0$ be a finitely generated field of definition of $X$. A class $u \in \rm H_{\textnormal{\'{e}t}}^{2i}(X, \ZZ_\ell(i))$ is an \textit{integral Tate class} if it is fixed by some open subgroup of $\Gal(k/k_0)$. Totaro has shown that for codimension-one cycles on $X$, the Tate conjecture over $k$ implies the integral Tate conjecture over $k$ \cite[Lemma 6.2]{totaroIHCthreefolds}. This means that every integral Tate class is the class of an algebraic cycle over $k$ with $\ZZ_\ell$-coefficients. 

Suppose that $A_{/k}$ is an abelian variety, defined over a finitely generated field $k_0 \subset k$ such that $k$ is the separable closure of $k_0$. Then the Tate conjecture for codimension-one cycles holds for $A$ over $k$ by results of Tate \cite{tate_endomorphisms}, Faltings \cite{Faltings1983, faltings_wustholz}, and Zarhin \cite{zarhin_one, zarhin_two}. By the above, $A$ satisfies the integral Tate conjecture for codimension-one cycles over $k$. On the other hand, the Fourier transform defines an isomorphism
\begin{equation} \label{eq:fourieretalecohomology}
\mr F_A \colon \rm H_{\textnormal{\'{e}t}}^2(A, \ZZ_\ell(1)) \xrightarrow{\sim} \rm H_{\textnormal{\'{e}t}}^{2g-2}(\wh A, \ZZ_\ell(g-1)),
\end{equation}
see \cite[Section 7]{totaroIHCthreefolds}. Since (\ref{eq:fourieretalecohomology}) is Galois-equivariant (the Poincar\'e bundle being defined over $k_0$) it sends integral Tate classes to integral Tate classes. Therefore, to prove the integral Tate conjecture for one-cycles on $A$, it suffices to lift (\ref{eq:fourieretalecohomology}) to a homomorphism $\CH^1(A)_{\ZZ_\ell} \to \CH_1(\wh A)_{\ZZ_\ell}$. 

\begin{proof}[Proof of Theorem \ref{introth:integraltate}]
This follows from the above together with Propositions \ref{prop:motivicetale-new} and
\ref{remarketalebetti}.\ref{remark3.12.2}.
\end{proof}
\noindent
For an abelian variety $X$ over the separable closure $k$ of a finitely generated field, call $\alpha \in \rm H^{2\bullet}_{\et}(X, \ZZ_\ell(\bullet))$ \emph{algebraic} if $\alpha$ is in the image of the cycle class map $\CH(X) \otimes \ZZ_\ell \to  \rm H^{2\bullet}_{\et}(X, \ZZ_\ell(\bullet))$. 

\begin{corollary}Let $A$ and $B$ be abelian varieties defined over the separable closure $k$ of a finitely generated field, of respective dimensions $g_A$ and $g_B$. 
\begin{enumerate}[wide, labelwidth=!, labelindent=0pt]
    \item \label{item:integraltateone} The classes $c_1(\ca P_A)^{2g_A-1}/(2g_A-1)!$ in $\rm H_{\textnormal{\'et}}^{4g_A-2}(A \times \wh A, \ZZ_\ell(2g_A-1))$ and $c_1(\ca P_B)^{2g_B-1}/(2g_B-1)! $ in $\rm H_{\textnormal{\'et}}^{4g_B-2}(B \times \wh B, \ZZ_\ell(2g_B-1))$ are algebraic if and only if $A \times \wh A$, $B \times \wh B$, $A\times B$ and $\wh A \times \wh B$ satisfy the integral Tate conjecture for one-cycles. 
    \item  \label{item:integraltatetwo} If $A$ and $B$ are principally polarized, then the integral Tate conjecture for one-cycles holds for $A\times B$ if and only if it holds for both $A$ and $B$. 
    \item \label{item:integraltatethree} Let $g = g_A$ and let $\theta \in \rm H^2_{\textnormal{\'et}}(A, \ZZ_\ell(1))$ be the first Chern class of an ample line bundle that induces a principal polarization on $A$. Suppose that $\theta^{g-1}/(g-1)! \in  \rm H_{\textnormal{\'{e}t}}^{2g-2}(A, \ZZ_\ell(g-1))$ is algebraic. Then for every algebraic cohomology class $\alpha \in   \oplus_{j > 0} \rm H_{\textnormal{\'{e}t}}^{2j}(A, \ZZ_\ell(j)) \subset \rm H_{\textnormal{\'{e}t}}^{2\bullet}(A, \ZZ_\ell(\bullet))$ and every $i \in \ZZ_{\geq 1}$, the cohomology class $\alpha^i/i! \in \rm H_{\textnormal{\'{e}t}}^{2\bullet}(A, \ZZ_\ell(\bullet))$ is algebraic. 
    \end{enumerate}
\end{corollary}

\begin{proof}
\ref{item:integraltateone}. See Equation (\ref{eq1}). 

\ref{item:integraltatetwo}. This is true because the minimal cohomology class of the product is algebraic if and only if the minimal cohomology classes of the factors are algebraic.

\ref{item:integraltatethree}. This follows from Propositions \ref{prop:motivicetale-new} and \ref{remarketalebetti}.\ref{remark3.12.2}. 
\end{proof}
\noindent
Combining Theorems \ref{maintheorem} and \ref{introth:integraltate}, we obtain:

\begin{corollary} \label{hodgetatecomparison}
Let $A_K$ be a principally polarized abelian variety over a number field $K\subset \CC$ and let $A_\CC$ be its base change to $\CC$. Then $A_\CC$ satisfies the integral Hodge conjecture for one-cycles if and only if $A_{\bar K}$ satisfies the integral Tate conjecture for one-cycles over $\bar K = \bar \QQ$.
\end{corollary}

\begin{proof}
We view $\bar \QQ$ as a subfield of $\CC$ in a way compatible with the inclusion $ K \hookrightarrow \CC$. For a prime number $\ell$, let $\theta_\ell \in \rm H_{\textnormal{\'et}}^{2}(A_{\bar \QQ}, \ZZ_\ell(1))$ be the $\ell$-adic \'etale cohomology class of the polarization of $A_{\bar \QQ}$. Similarly, define $\theta_\CC \in \NS(A_\CC) \subset \rm H^{2}(A_\CC, \ZZ)$ as the polarization of the complex abelian variety $A_\CC$. By Theorems \ref{maintheorem} and \ref{introth:integraltate}, it suffices to show that $\gamma_{\theta_\CC} \in \rm H^{2g-2}(A_\CC, \ZZ)$ is algebraic if and only if $\gamma_{\theta_\ell} \in \rm H_{\textnormal{\'et}}^{2g-2}(A_{\bar \QQ}, \ZZ_\ell(g-1))$ is in the image of (\ref{eq:integraltateconjecture}) for each prime number $\ell$. 

The Artin comparison theorem gives an isomorphism of $\ZZ_\ell$-algebras
\[
\phi\colon \rm H^{\bullet}_{\textnormal{\'et}}(A_{\bar \QQ}, \ZZ_\ell) = \rm H^{\bullet}_{\textnormal{\'et}}(A_{\CC}, \ZZ_\ell) \cong \rm H^{\bullet}(A_\CC, \ZZ) \otimes_\ZZ \ZZ_\ell.
\]
Since $\phi$ is compatible with the cycle class maps $cl_{\bar \QQ} \colon \CH(A_{\bar \QQ}) \to \rm H^{\bullet}_{\textnormal{\'et}}(A_{\bar \QQ}, \ZZ_\ell)$ and $cl_\CC \colon \CH(A_\CC) \to \rm H^{\bullet}(A_\CC, \ZZ)$, we have $\phi(\gamma_{\theta_\ell}) = \gamma_{\theta_\CC}$. 
Define $$\rm R^{2g-2}(A) = \Coker\left(\CH_1(A_\CC) \to \rm H^{2g-2}(A_\CC,\ZZ)\right).$$ Then $\rm R^{2g-2}(A)\otimes \ZZ_\ell = \Coker\left(\CH_1(A_\CC)_{\ZZ_\ell} \to \rm H^{2g-2}(A_\CC,\ZZ_\ell)\right)$. 
Suppose that $\gamma_{\theta_\ell}$ is in the image of (\ref{eq:integraltateconjecture}) for every prime number $\ell$. 
The image of $\gamma_{\theta_\CC}$ in $\rm R^{2g-2}(A)\otimes \ZZ_\ell$ is then zero for each prime $\ell$, which implies that the image of $\gamma_{\theta_\CC}$ in $\rm R^{2g-2}(A)$ is zero, i.e. $\gamma_{\theta_\CC}$ is algebraic. Conversely, suppose that $\gamma_{\theta_\CC} = \sum_{i= 1}^k n_i \cdot cl(C_i)$ for some smooth projective curves $C_i$ over $\CC$. 
The Hilbert scheme $\ca H = \textnormal{Hilb}_{A_K/K}$ is defined over $K$; 
for each $i = 1, \dotsc, k$ we pick a $\bar \QQ$-point in the connected component of $\ca H$ containing $[C_i \subset A]$. This gives smooth projective curves $C_i' \subset A_{\bar \QQ}$ over $\bar \QQ$. 
If $\Gamma = \sum_i n_i \cdot [C_i'] \in \CH_1(A_{\bar \QQ})$, then we have $cl_\CC(\Gamma_\CC) = \gamma_{\theta_\CC}$ by Lemma \ref{lemma_algebraicityminimalclass}, hence $cl_{\bar \QQ}(\Gamma) = \gamma_{\theta_\ell}$. 
\end{proof}

\noindent
Another corollary of Theorem \ref{introth:integraltate} is that the integral Tate conjecture for one-cycles on principally polarized abelian varieties is stable under specialization. Indeed, one has (c.f. Corollary \ref{complexspecialization}):

\begin{corollary} \label{cor:specialization}
Let $A_K$ be a principally polarized abelian variety over a number field $K$ and suppose that $A_{\bar K}$ satisfies the integral Tate conjecture for one-cycles over $\bar K$. 
Let $\mf p$ be a prime ideal of the ring of integers $\OO_K$ of $K$ at which $A_K$ has good reduction and write $\kappa = \OO_K/\mf p$. Then the abelian variety $A_{\bar \kappa}$ over $\bar \kappa$ satisfies the integral Tate conjecture for one-cycles over $\bar \kappa$. 
\end{corollary}

\begin{proof}
Write $S = \Spec \OO_K$ and let $A \to S$ be the N\'eron model of $A_K$. Let $R$ (resp.\ $K_{\mf p}$) be the completion of $\OO_K$ (resp.\ $K$) at the prime $\mf p$. The natural composition $K \to K_{\mf p} \to \bar K_{\mf p}$ induces an embedding $\bar K \to \bar K_{\mf p}$, where $\bar K_{\mf p}$ is an algebraic closure of $K_{\mf p}$. This gives a commutative diagram, where the square on the right is provided in \cite[Example 20.3.5]{fultonintersection}:
\begin{equation} \label{specializationdiagram}
\begin{split}
    \xymatrix{
\xymatrixcolsep{3pc}
\CH(A_{\bar K})_{\ZZ_\ell} \ar[r] \ar[d] & 
\CH(A_{\bar K_{\mf p}})_{\ZZ_\ell} \ar[r] \ar[d] &
\CH(A_{\bar \kappa})_{\ZZ_\ell} \ar[d]\\ 
\oplus_{r \geq 0}\rm H^{2r}_{\textnormal{\'et}}(A_{\bar K}, \ZZ_\ell(r))\ar[r]^{\sim} &\oplus_{r \geq 0}\rm H^{2r}_{\textnormal{\'et}}(A_{\bar K_{\mf p}}, \ZZ_\ell(r))\ar[r]^{\sim}&
\oplus_{r \geq 0}\rm H^{2r}_{\textnormal{\'et}}(A_{\bar \kappa}, \ZZ_\ell(r)).
}
\end{split}
\end{equation}
Now the principal polarization $\lambda_K\colon A_K \xrightarrow{\sim} \wh A_K$ extends uniquely to a homomorphism $\lambda\colon A \to \wh A$ by the N\'eron mapping property \cite[Section 1.2, Definition 1]{BLR} and since the same is true for the inverse $\lambda_K^{-1}\colon \wh A_K \xrightarrow{\sim} A_K$ we find that $\lambda$ is an isomorphism. In particular, we see that $A_{\bar \kappa}$ is principally polarized and that the class in $\CH^1(A_{\bar K})_{\ZZ_\ell}$ of a theta divisor on $A_{\bar K}$ is sent to the class in $\CH^1(A_{\bar \kappa})_{\ZZ_\ell}$ of a theta divisor on $A_{\bar \kappa}$. Thus, the minimal class $\gamma_{\theta_{\bar K}} \in \rm H^{2g-2}_{\textnormal{\'et}}(A_{\bar K}, \ZZ_\ell(g-1))$ is sent to the minimal class $\gamma_{\theta_{\bar \kappa}} \in \rm H^{2g-2}_{\textnormal{\'et}}(A_{\bar \kappa}, \ZZ_\ell(g-1))$ by the isomorphism on the bottom of Diagram (\ref{specializationdiagram}). It follows that $\gamma_{\theta_{\bar \kappa}}$ is algebraic which by Theorem \ref{introth:integraltate} means that we are done.  
\end{proof}

Finally, let us prove Theorem \ref{th:chinglichai}. The theorem follows from Theorem \ref{introth:integraltate} together with a result of Chai on the density of an ordinary isogeny class in positive characteristic \cite{chaiordinaryhecke}.

\begin{proof}[Proof of Theorem \ref{th:chinglichai}]
For any $t \in A_g(k)$, let $(A_t, \lambda_t)$ be a principally polarized abelian variety such that $[(A_t, \lambda_t)] = t$. Let $A = E_1 \times \cdots \times E_g$ be the product of $g$ ordinary elliptic curves $E_i$ over $k$ and provide $A$ with its natural principal polarization. Let $x \in \msf A_{g}(k)$ be the point corresponding to the isomorphism class of $A$. Let $q > (g-1)!$ be a prime number different from $p$ and let $\mr G_q(x) \subset \msf A_{g}( k)$ be the set of isomorphism classes $y = [(A_y,\lambda_y)]$ that admit an isogeny $\phi \colon A_y \to A_x$ with $\phi^\ast\lambda_x = q^N \cdot \lambda_y$ for some nonnegative integer $N$. We claim that $A_y$ satisfies the integral Tate conjecture for one-cycles over $k$ for any $y \in \mr G_q(x)$. Indeed, for such $y$ there exists a nonnegative integer $N$ such that the isogeny $[q^N]\colon A_y \to A_y$ factors through $A_x$. Consequently, 
$q^{(2g-2)\cdot N} \cdot \gamma_\theta$ is algebraic for the first Chern class $\theta$ of the principal polarization on $A_y$, which implies that $\gamma_\theta$ is algebraic (as $q > (g-1)!$). Thus, the claim follows from Theorem \ref{introth:integraltate}. Now $\mr G_q(z)$ is dense in $\msf A_g$ for any ordinary principally polarized abelian variety $(A_z,\lambda_z)$ by a result of Chai \cite[Theorem 2]{chaiordinaryhecke}. Therefore, $\mr G_q(x)$ is dense in $\msf A_g$ and the proof is finished.
\end{proof}

\printbibliography

\textsc{Thorsten Beckmann, Mathematical Institute of the University of Bonn, University of Bonn, Endenicher Allee 60, 53115 Bonn, Germany}\par\nopagebreak
  \textit{E-mail address:} \texttt{beckmann@math.uni-bonn.de}
\\
\\
\textsc{Olivier de Gaay Fortman, Institute of Algebraic Geometry, Leibniz University Hannover, Welfengarten 1, 30167 Hannover, Germany}\par\nopagebreak
  \textit{E-mail address:} \texttt{degaayfortman@math.uni-hannover.de}

\end{document}